\documentclass[12pt]{amsart}
\usepackage{amsthm,amsmath,a4,amssymb,verbatim,url}

\newtheorem{thm}{Theorem}[section]
\newtheorem{cor}[thm]{Corollary}
\newtheorem{lem}[thm]{Lemma}
\newtheorem{prop}[thm]{Proposition}
\theoremstyle{definition}
\newtheorem{defn}[thm]{Definition}
\theoremstyle{remark}
\newtheorem{rem}[thm]{Remark}
\newtheorem{que}[thm]{Question}
\newtheorem{ex}[thm]{Example}
\numberwithin{equation}{section}

\newcommand{\Aut}{\textnormal{Aut}}
\newcommand{\sys}{\mathrm{sys}}
\newcommand{\lf}{\lfloor}
\newcommand{\rf}{\rfloor}
\newcommand{\cov}{\mathrm{cov}}
\newcommand{\g}{\mathfrak{g}}
\newcommand{\mk}{\mathfrak}

\begin{document}

\address{CNRS -- D\'epartement de Math\'ematiques, Universit\'e Paris-Sud, 91405 Orsay, France}

\email{yves.cornulier@math.u-psud.fr}
\thanks{This material is based upon work supported by the NSF under Grant No.\
DMS-1440140 while the author was in residence at the MSRI (Berkeley)
during the Fall 2016 semester. Supported by ANR Project ANR-14-CE25-0004 GAMME}



\title{Nilpotent Lie algebras and systolic growth of nilmanifolds}
\author{Yves Cornulier}%
\date{April 26, 2017}


\begin{abstract}
Introduced by Gromov in the nineties, the systolic growth of a Lie group gives the smallest possible covolume of a lattice with a given systole.

In a simply connected nilpotent Lie group, this function has polynomial growth, but can grow faster than the volume growth. We express this systolic growth function in terms of discrete cocompact subrings of the Lie algebra, making it more practical to estimate. 

After providing some general upper bounds, we develop methods to provide nontrivial lower bounds. We provide the first computations of the asymptotics of the systolic growth of nilpotent groups for which this is not equivalent to the volume growth. In particular, we provide an example for which the degree of growth is not an integer; it has dimension 7. Finally, we gather some open questions.
\end{abstract}
\maketitle

\section{Introduction}

\subsection{Background}

Every locally compact group $G$ has a Haar measure $\mu$, unique up to positive scalar multiplication. If in addition $G$ is generated by a symmetric compact neighborhood $S$ of 1, the function $b(n)=\mu(S^n)$ is called the volume growth (or word growth) of $G$ ; while its values depend on the choice of $(S,\mu)$, its asymptotics (in the usual meaning, recalled in \S\ref{s_asy}) does not. The volume growth is either exponential or subexponential. Those compactly generated locally compact group with polynomially bounded growth have been characterized by Guivarch and Jenkins \cite{Guiv,Jen} in the case of connected Lie groups, Gromov in the case of discrete groups \cite{Gro81}, and Losert \cite{Los} in general. All such groups are commable, and hence quasi-isometric, to simply connected nilpotent Lie groups, and thus, by work of Guivarch \cite{Guiv} have an integral degree of polynomial growth that is easily computable in terms of the Lie algebra structure (see \S\ref{pre_la}).

The object of study of the paper is the following related notion of growth, introduced by Gromov in \cite[p. 333]{Gro}.

\begin{defn}\label{d_sy}Let $H$ be a locally compact group and $|\cdot|$ the word length relative to some choice of compact generating subset. If $X\subset H$, define its systole as $\sys(X)=\inf\{|x|:x\in X\smallsetminus\{1\}\}\in\mathbf{R}_+\cup\{+\infty\}$.

Endow $H$ with some left-invariant Haar measure. The systolic growth of $(H,|\cdot|)$ is the function mapping $r\ge 0$ to the infimum $\sigma(r)\in R_+\cup\{+\infty\}$ of covolumes of cocompact lattices of $H$ with systole $\ge r$.
\end{defn}

See Remark \ref{riemannian} for the geometric interpretation in a Riemannian setting.
Note that the definition makes sense when $H$ is discrete, in which case lattices just refer to finite index subgroups: this is actually the setting in Gromov's original definition. In the setting we will focus on, $H$ will always be nilpotent and in this case all lattices are cocompact. In general, we can define another type of systolic growth, allowing non-cocompact lattices.

The asymptotics of the growth of $\sigma$ does not depend on the choice of the word length. The number $\sigma(r)$ is always bounded below by the volume of the open $r/2$-ball in $H$.

The function $\sigma$ is interesting only when it takes finite values, in which case we say that $H$ is residually systolic. When $H$ is discrete, this just means that $H$ is residually finite. In general, a sufficient condition for $H$ being residually systolic is that $H$ admits a residually finite cocompact lattice. 

It is natural to compare the volume growth and the systolic growth. For finitely generated linear groups of exponential growth, the systolic growth is exponential as well \cite{BC}.

\subsection{Background with focus in the nilpotent case}

Given a Lie algebra $\g$, denote by $(\g^i)_{i\ge 1}$ its lower central series (see \S\ref{pre_la}); by definition $\g$ is $c$-step nilpotent if $\g^{c+1}=\{0\}$. The homogeneous dimension of $\g$ is classically defined as the sum \[D=D(\g)=\sum_{i\ge 1}\dim(\g^i);\] we have $D<\infty$ if and only $\g$ is nilpotent and finite-dimensional.

A classical result of Malcev is that a simply connected nilpotent Lie group admits a lattice (which is then cocompact and residually finite) precisely when its Lie algebra can be obtained from a rational Lie algebra by extension of scalars. Therefore this is also equivalent to being residually systolic. In this case, the systolic growth is easy to bound polynomially; nevertheless the comparison between volume growth and the systolic growth is not obvious, because the precise rate of polynomial growth is an issue. 
 
A first step towards a good understanding is the following result (all asymptotic results are meant when $r\to +\infty$).

\begin{thm}[\cite{gralie}]\label{carnoy} Let $G$ be a simply connected nilpotent Lie group with a lattice $\Gamma$. Let $\g$ be the Lie algebra of $G$, and let $D$ be its homogeneous dimension (so, for both $G$ and $\Gamma$, the growth is $\simeq r\mapsto r^D$ and the systolic growth is $\succeq r^D$, see \S\ref{pre_la}). The following are equivalent:
\begin{enumerate}
\item\label{i_gd} the systolic growth of $G$ is $\simeq r^D$;
\item the systolic growth of $\Gamma$ is $\simeq r^D$;
\item\label{i_gc} $\g$ is Carnot, i.e., admits a Lie algebra grading $\g=\bigoplus_{i\ge 1}\g_i$ such that $\g_1$ generates $\g$.
\end{enumerate}
Otherwise, if $\sigma$ denotes the systolic growth of either $\Gamma$ or $G$, then it satisfies $\sigma(r)\gg r^D$.
\end{thm}

In the non-Carnot case, the proof of Theorem \ref{carnoy} does not provide any explicit asymptotic lower bound improving $\sigma(r)\gg r^D$. In this paper, we carry out the task of evaluating the systolic growth in a number of explicit non-Carnot cases. Such results are presented in \S\ref{compu}. We start with some general upper bounds. We will often emphasize the quotient $\sigma'(r)=r^{-D}\sigma(r)$, since it often naturally occurs in computations, and its growth is a measure of the failure of being Carnot.

\subsection{Upper bounds}\label{s_iub}

We provide here some upper bounds on the systolic growth. 
 We denote by $\lceil\cdot\rceil$ the ceiling function. Given $c\ge 0$, we define
\[k_c(\g)=\sum_{i=1}^{\lceil c/2\rceil-1}\left(\frac{c}2-i\right)\dim(\g^i/\g^{i+1});\]
Note that $k_c(\g)\le\left(\frac{c}2-1\right)\dim(\g/\g^{\lceil c/2\rceil})$. 

\begin{prop}[See Proposition \ref{ineq}]
Let $\g$ be a finite-dimensional $c$-step nilpotent real Lie algebra with homogeneous dimension $D$, and let $k=k_c(\g)$ be defined as above. Assume that the corresponding simply connected nilpotent Lie group $G$ admits lattices. Then the systolic growth $\sigma(r)$ of $G$ and its lattices is $\preceq r^{D+k}$.
\end{prop}

We have $k_c(\g)\le\frac16\dim(\g)^2$ (see Proposition \ref{ub_kc}) when $c$ is the nilpotency length of $\g$, so that we obtain an upper bound on $\sigma(r)/r^D$ depending only on $\dim(\g)$. For small values of $c$, we have
\[k_{\le 2}(\g)=0;\quad k_3(\g)=\frac12\dim(\g/\g^2),\quad k_4(\g)=\dim(\g/\g^2);\] \[k_5(\g)=\frac32\dim(\g/\g^2)+\frac12\dim(\g^2/\g^3);\quad k_6(\g)=2\dim(\g/\g^2)+\dim(\g^2/\g^3).\]
Also note that $D(\g)+k_c(\g)\le c\dim(\g)/2$. This improves the trivial upper bound $\preceq r^{c\dim(\g)}$, making use of congruence subgroups in a lattice, which was mentioned in \cite{gralie}.

Note that every 2-step nilpotent Lie algebra is Carnot; the smallest nilpotency length allowing non-Carnot Lie algebra is 3. When $\g$ is 3-step nilpotent, the above proposition yields $\sigma(r)\preceq r^{D+\dim(\g/\g^2)/2}$. This bound is not very far from sharp, see Theorem \ref{ig3nil}.

\subsection{Lower bounds on the systolic growth and precise estimates}\label{compu}

This part is the bulk of the paper. It contains the first exact estimates of the asymptotic behavior of the systolic growth of nilpotent Lie groups beyond the Carnot case covered by Theorem \ref{carnoy}. The first non-Carnot Lie algebras occur in dimension 5 and in this case we have the following theorem.

\begin{thm}\label{t_5dim}
Let $G$ be a 5-dimensional simply connected nilpotent Lie group whose Lie algebra $\g$ is non-Carnot (there are 2 non-isomorphic possibilities for $\g$, for which the homogeneous dimension $D$ is either 8 or 11). Then the systolic growth of both $G$ and its lattices is $\simeq r^{D+1}$. 
\end{thm}

Both cases are obtained in a single proof. In dimension 6, the classification yields 13 non-Carnot nilpotent real Lie algebras; a similar approach provides precise estimates for at least some of them, but I do not know if it can exhibit a behavior different from being $\simeq r^{D+h}$ with $h\in\{1,2,3\}$. In dimension 7, where a classification is still known (but lengthy), a similar approach yields an example for which the degree is not an integer:

\begin{thm}\label{7nonin}
There exists a 7-dimensional simply connected nilpotent Lie group for which the systolic growth, as well as the systolic growth of one of its lattices, is $\simeq r^{D+3/2}$.
\end{thm}

This contrasts with the fact that the volume growth always has an integral degree of polynomial growth (the homogeneous dimension $D$). Yet so far we only know, for the systolic growth, behaviors of the form $r^{D+h}$ with $h$ a non-negative rational, but we actually do not know if $\log\sigma(r)/\log(r)$ always converges, and if so, if its limit is always a rational, and what kind of further constraints we can expect on $h$ (see the questions below).

At the computational level, let us also provide some families of unbounded dimension, for which we obtain unbounded values for $h$.

\begin{thm}[Truncated Witt Lie algebra]\label{vira}
For $n\ge 3$, let $G(n)$ be the simply connected nilpotent Lie group corresponding to the Lie algebra $\g(n)$ with basis $(e_i)_{1\le i\le n}$ and nonzero brackets $[e_i,e_j]=(i-j)e_{i+j}$, ($i+j\le n$). Then its systolic growth grows as $r\mapsto r^{D+h}$ with $h=\lceil (n-4)/2\rceil$ (here the homogeneous dimension is $D=\frac{n(n-1)}{2}+1$). 
\end{thm}

The following family of examples with unbounded $h$ consists of 3-step nilpotent Lie groups.

\begin{thm}[see Theorem \ref{g3nil}]\label{ig3nil}
For $n\ge 0$, let $\g(4+2n)$ be the 3-step nilpotent $(4+2n)$-dimensional Lie algebra obtained as central product of a 4-dimensional filiform Lie algebra and a $(2n+1)$-dimensional Heisenberg Lie algebra, and $G(4+2n)$ the corresponding simply connected nilpotent Lie group. Then its systolic growth grows as $r\mapsto r^{D+n}$, where $D=2n+7$ is the homogeneous dimension.
\end{thm}

The same method actually yields examples for which the polynomial degree of $r^{-D}\sigma(r)$ is comparable to the square of the dimension, see Remark \ref{quadra_h}.

\subsection{Outline of the method}

Let us outline the method used to obtained the estimates of \S\ref{compu}.

The first step is to translate the problem, which concerns lattices in a simply connected nilpotent Lie group, into a problem about discrete cocompact subrings in its Lie algebra. This uses the fact that even if there is no exact correspondence between the two (the exponential of a Lie subring can fail to be a subgroup, and the logarithm of a subgroup can fail to be a Lie subring), this is true ``up to bounded index". This important fact, for which we claim no originality but could not refer to a written proof, is formulated in Lemma \ref{llazard} and proved (along with a more precise statement) in the appendix. 

Then the notion of systolic growth can be made meaningful in a real finite-dimensional nilpotent Lie algebra: it maps $r$ to the smallest covolume of a discrete cocompact subring of Guivarch systole $\ge r$. Here the Guivarch systole is the Lie algebra counterpart of the systole: this is the smallest Guivarch length of a nontrivial element in the lattice. The Guivarch length is recalled in \S\ref{pre_la}; for instance, in the 3-dimensional Heisenberg Lie algebra, the Guivarch length of an element $\begin{pmatrix}0 & x & z\\ 0 & 0 & y\\ 0 & 0 & 0\end{pmatrix}$ can be defined as the value $|x|+|y|+|z|^{1/2}$. See \S\ref{syrnla}. The previous fact shows that the systolic growth of a simply connected nilpotent Lie group is asymptotically equivalent to that of its Lie algebra. 

Next, we have to estimate the systolic growth in various Lie algebras. The idea is to use a flag of rational ideals $\g=\mk{w}_1\ge \dots \dots\mk{w}_k=\{0\}$. Here, if we consider arbitrary lattices, we need these ideals to be rational for every rational structure (we call this solid and provide some basic fact about such ideals in \S\ref{s_solid}). For instance, terms of the lower central series are such ideals. Then any lattice intersects each $\mk{w}_i$ into a lattice and this intersection maps into a lattice in $\mk{w}_i/\mk{w}_{i+1}$; let $a_i$ be the corresponding covolume. Then the covolume of the whole lattice is $\prod_ia_i$. Then we use the stability under brackets and the hypothesis of Guivarch systole $\ge r$ to obtain lower bounds on $\prod a_i$, which in some cases are better than the trivial lower bound (the trivial lower bound has the form $\succeq r^D$).

More precisely, this approach typically yields, for a lattice of Guivarch systole $\ge r$, some inequalities of the form $a_ia_j\ge a_k r^{m(i,j,k)}$ for some integer $(i,j,k)$. If we write $A_i=\log_r(a_i)$ (so that the covolume is $r^{\sum A_i}$), this can be rewritten as $A_i+A_j\ge A_k+m(i,j,k)$. Then such a family of inequalities can yield a lower bound of the form $\sum A_i\ge q$ for some rational $q$, and thus yielding a lower bound for the covolume $\prod a_i \ge r^q$. 

Once such a method is checked to yield precise estimates in some cases, it is not a surprise to find that in well-chosen examples, it yields non-integral degrees, as in Theorem \ref{7nonin}.

Let us also mention that we actually renormalize the problem by a well-chosen family of linear automorphisms (\S\ref{s_dil}), which yields lower bounds for $r^{-D}\sigma(r)$ and simplifies the computations (for instance, it allows to treat simultaneously both examples of Theorem \ref{t_5dim}). The approach also provides a new, simpler proof of the implication (\ref{i_gd})$\Rightarrow$(\ref{i_gc}) in Theorem \ref{carnoy}, see \S\ref{s_small}.

\subsection{Open questions}
Let us mention some open problems. Since the number $h$ is possibly not always defined, we write things as follows. Let $H$ be a compactly generated, locally compact group of polynomial growth, admitting at least a lattice (we especially have in mind the cases when $H$ is a simply connected nilpotent Lie group, or $H$ is a finitely generated nilpotent group). So its systolic growth $r\mapsto\sigma(r)$ is well-defined. Let $D$ be its homogeneous dimension. Define 
\[\underline{h_H}=\underline{\lim}\frac{\log(\sigma(n)/n^D)}{\log(n)};\quad\overline{h_H}=\overline{\lim}\frac{\log(\sigma(n)/n^D)}{\log(n)}\]

\begin{que}\label{conj_infsup}~
\begin{enumerate}
\item\label{conjis}
Is it always true that $\underline{h_H}=\overline{h_H}$? (I conjecture a positive answer).
\item Are $\underline{h_H}$ and $\overline{h_H}$ always rational numbers?
\item Is it true that $\sigma_H(n)\simeq n^{D+h'}$ for some $h'\ge 0$? (Of course this implies a positive answer to (\ref{conjis}), but this is more optimistic and I do not conjecture anything.)
\end{enumerate}
\end{que}

\begin{que}
Does there exist infinitely many non-equivalent types of systolic growth asymptotically bounded above by some give polynomial?
\end{que}

\begin{que}\label{q_GamG}
Let $G$ be a simply connected nilpotent Lie group with a lattice $\Gamma$, with systolic growth $\sigma_G$ and $\sigma_\Gamma$ (so $\sigma_G\preceq\sigma_\Gamma$).

\begin{enumerate}
\item Do we always have $\sigma_G\simeq\sigma_\Gamma$? 
\item We now refer to the uniform systolic growth introduced in \S\ref{s_usg}. We have $\sigma_G\preceq\sigma^u_G\preceq\sigma^u_{\Gamma,G}$. Do we always have $\sigma_G\simeq\sigma^u_G$? Do we always have $\sigma^u_G\simeq\sigma^u_{\Gamma,G}$?
\end{enumerate}
\end{que}

\noindent{\bf Acknowledgements.} I thank Yves Benoist for useful hints, and Pierre de la Harpe for corrections on a preliminary version of this paper. I thank the referee for various corrections and useful references.

\setcounter{tocdepth}{1}
\tableofcontents

\section{Algebraic preliminaries}

In all this section, $K$ is a field of characteristic zero, unless explicitly specified.

\subsection{Lie algebras: lower central series, growth}\label{pre_la}

Let $\g$ be a nilpotent fd Lie algebra over $K$ (fd stands for finite-dimensional). Its lower central series is defined by $\g^1=\g$, $\g^{i+1}=[\g,\g^i]$ for $i\ge 1$. Let $\mk{v}_i$ be a supplement subspace of $\g^{i+1}$ in $\g^i$, so that $\g=\bigoplus_{i\ge 1}\mk{v}_i$ (and $v_i=0$ for large $i$). We have, for all $i,j$
\[ [\mk{v}_i,\mk{v}_j]\subset \bigoplus_{k\ge i+j}\mk{v}_k,\qquad \mk{v}_{i+1}\subset [\mk{v}_1,\mk{v}_i]+\g^{i+2}.\]


Let $G$ be the group of $K$-points of the corresponding unipotent algebraic group; $G$ can obtained from $\g$ using the Baker-Campbell-Hausdorff formula as group law. When $K=\mathbf{R}$, this is the simply connected Lie group associated to $\g$.

The integer
\[D=D(\g)=D(G)=\sum i\dim(\mk{v}_i)=\sum_i\dim\g^i\]
is called the homogeneous dimension of $G$.
Indeed, in the real case, the volume of the $r$-ball is $\simeq r^D$ \cite{Guiv} (see \S\ref{s_asy} for the definition of $\simeq$); this degree formula was also found by Bass \cite{Bas} while restricting to the discrete setting. We have $D\ge\dim(\g)$, with equality if and only if $\g$ is abelian.

Again in the real case, fix a norm on each $\mk{v}_i$. If $x=(x_1,x_2,\dots)\in\g$ in the decomposition $\g=\bigoplus_{i\ge 1}\mk{v}_i$, define its Guivarch length $\lf x\rf=\sup\|x_i\|^{1/i}$. This length plays an important role, as Guivarch established that $\lf x\rf$ is a good estimate for the word length of $\exp(x)$ in the simply connected nilpotent Lie group $G$ associated to $\g$. 

\subsection{Solid ideals}\label{s_solid}

We introduce here the notion of solid ideals, which will be useful when computing lower bounds on the systolic growth of simply connected nilpotent Lie groups.

Let $\g$ be a Lie algebra (over $K$). A $\mathbf{Q}$-structure on $\g$ is the data of a $\mathbf{Q}$-subspace $\mk{l}$ such that the canonical homomorphism $j:\mk{l}\otimes_\mathbf{Q} K\to\g$ is a linear isomorphism of Lie $K$-algebras. If $\mk{l}$ is a $\mathbf{Q}$-subalgebra, we call it a multiplicative $\mathbf{Q}$-structure, and then $j$ is a $K$-algebra isomorphism.

Given a $\mathbf{Q}$-structure $\mk{l}$, a $K$-subspace $V$ of $\g$ is called $\mathbf{Q}$-defined if it is generated as a $K$-subspace by $V\cap\mk{l}$. If $\g$ is a Lie $K$-algebra, we say that an ideal is solid if it is $\mathbf{Q}$-defined for every multiplicative $\mathbf{Q}$-structure.
The following properties are straightforward.

\begin{itemize}
\item A ideal contained and solid in a solid ideal is solid in the whole Lie algebra;
\item the inverse image of a solid ideal by the quotient by a solid ideal is solid;
\item the bracket of two solid ideals is solid:
\item the centralizer of a solid ideal is solid;
\item the intersection and the sum of two solid subalgebras are solid.
\end{itemize} 

For instance, if $\g$ is abelian, then the only solid ideals of $\g$ are $\{0\}$ and $\g$. 

To single out solid ideals in fd $\mathbf{Q}$-definable real nilpotent Lie algebras $\g$ is useful in view of the following: an ideal $V\subset\g$ is solid if and only if for every discrete cocompact subring $\Lambda$, the intersection $\Lambda\cap V$ is a lattice in $V$ (or equivalently, the projection on $\g/V$ is a lattice in $\g/V$). The reason is that the $\mathbf{Q}$-linear span of any discrete cocompact subring is a multiplicative $\mathbf{Q}$-structure, and that conversely any multiplicative $\mathbf{Q}$-structure contains a discrete cocompact subring. 

Solid ideals can almost be recognized using how they behave under the automorphism group.
Say that an ideal $I$ in $\g$ is absolutely $\Aut$-invariant if, denoting by $\bar{K}$ an algebraic closure of $K$ and $\Aut(\g)_{\bar{K}}$ for the group of automorphisms of the $\bar{K}$-algebra $\g\otimes_K\bar{K}$, the ideal $I\otimes_K\bar{K}$ is $\Aut(\g)_{\bar{K}}$-invariant.

\begin{thm}\label{t_solid}
Assume that $K$ is uncountable. Let $\g$ be a fd Lie algebra over $K$. Then
\begin{enumerate}
\item every solid ideal is invariant under $\Aut(\g)^0$ (or equivalently, stable under all $K$-linear self-derivations of $\g$);
\item every absolutely $\Aut$-invariant ideal is solid.
\end{enumerate}
\end{thm}

\begin{rem}
If $I$ is an ideal of $\g$, we have: $I$ absolutely Aut-invariant $\Rightarrow$ $I$ $\Aut(\g)$-invariant $\Rightarrow$ $I$ $\Aut(\g)^0$-invariant ($H^0$ denoting the connected component of the unit in the Zariski topology). The reverse implications do not hold in general. For instance, in $\mk{sl}_2(K)\times\mk{sl}_2(K)$, the ideal $\mk{sl}_2(K)\times\{0\}$ only satisfies the third condition. Also, over the reals, in $\mk{sl}_2(\mathbf{R})\times\mk{so}_3(\mathbf{R})$, the ideal $\mk{sl}_2(\mathbf{R})\times\{0\}$ only satisfies the latter two conditions (however, it is solid). Over the complex numbers, I do not know whether every solid ideal is $\Aut(\g)$-invariant.
\end{rem}

Theorem \ref{t_solid} follows from the next two propositions. Since the Lie algebras axioms plays no role here, we consider arbitrary algebras and the context could be even more general. Also, solid can be defined for arbitrary subspaces and we use this straightforward generalization. However, in a Lie algebra solid subspaces are always ideals (Corollary \ref{solsubide})

\begin{prop}
Let $\g$ be a fd $K$-algebra and $I$ a $K$-subspace, $\mathbf{Q}$-defined for at least one multiplicative $\mathbf{Q}$-structure; let $\bar{K}$ be an algebraically closed extension of $K$. If $I\otimes_K\bar{K}$ is $\Aut(\g)_{\bar{K}}$-invariant and $\mathbf{Q}$-defined for some multiplicative $\mathbf{Q}$-structure on the algebra $\g$, then $I$ is solid.
\end{prop}

\begin{proof}
Let $\Xi$ be the Galois group of $\bar{K}$ over $\mathbf{Q}$. 

Let $V$ be a finite-dimensional $K$-vector space. Let $W\subset V$ be a $\mathbf{Q}$-structure in $V$. Then $\Xi$ acts coordinate-wise on $V_{\bar{K}}=V\otimes_K\bar{K}=W\otimes_\mathbf{Q}\bar{K}$ (for some choice of basis of $W$, whose choice does not matter).
We denote this action as $\gamma\cdot v=u_W(\gamma)v$ for $\gamma\in\Xi$. Then $u_W(\gamma)$ is a $\mathbf{Q}$-linear automorphism of $V_{\bar{K}}$; it is also $\gamma$-semi-linear, in the sense that $u_W(\gamma)(\lambda v)=\gamma(\lambda)u_W(\gamma)v$. It follows in particular that if $W'$ is another $\mathbf{Q}$-structure, then $\eta_{W,W'}(\gamma)=u_W^{-1}(\gamma)u_{W'}(\gamma)$ is a $\bar{K}$-linear automorphism of $V_{\bar{K}}$.

Now suppose that the action $u_W$ of $\Xi$ leaves a $\bar{K}$-subspace $J\subset V_{\bar{K}}$ invariant. This implies that $J$ is the $\bar{K}$-linear span of $J\cap W$ (see \cite[{\it corollaire} in Chap.\ V.4]{BouA}).

%
%
We apply this to $J=I_{\bar{K}}=I\otimes_K\bar{K}$. By assumption, $I$ is $\mathbf{Q}$-defined for some multiplicative $\mathbf{Q}$-structure $W'$. So $J$ is $u_{W'}(\Xi)$-invariant. It is also $\Aut(\g)(\bar{K})$-invariant, by assumption. Hence, by the first paragraph of the proof, it is $u_W(\Xi)$-invariant. So $I_{\bar{K}}$ is the $\bar{K}$-linear span of $J\cap W$. This means that $\dim_\mathbf{Q}(J\cap W)=\dim_{\bar{K}}(I_{\bar{K}})$. Since the latter equals $\dim_K(I)$ and since $J\cap W=I\cap W$, this in turn implies that $I$ is the $K$-linear span of $I\cap W$, that is, $I$ is $\mathbf{Q}$-defined. 
\end{proof}

\begin{prop}
Let $\g$ be a fd algebra over an uncountable field of characteristic zero, definable over $\mathbf{Q}$. Let $V$ be a subspace of $\g$. Suppose that $V$ is solid, that is, it is $\mathbf{Q}$-defined for every $\mathbf{Q}$-structure on $\g$. Then $V$ is invariant under $H=\Aut(\g)^0$, that is, when $K$ has characteristic zero, $V$ is stable under every derivation of $\g$.
\end{prop}
\begin{proof}
Let $L\subset H$ be the stabilizer of $V$. If $L\neq H$, then $H(K)/L(K)$ is uncountable: if $K$ is $\mathbf{R}$ or $\mathbf{C}$ this is because it is a manifold of positive dimension; in general see Lemma \ref{uncountable}. The map $hL\to hV$ from $H(K)/L(K)$ to the set of subspaces of $\g$ being injective, it has an uncountable image. So, given $\mathbf{Q}$-structure $\g_\mathbf{Q}$, there exists $h\in H(K)$ such that $hV$ is not $\mathbf{Q}$-defined. Accordingly, for the new $\mathbf{Q}$-structure defined by $h^{-1}\g_\mathbf{Q}$, $V$ is not $\mathbf{Q}$-defined, contradicting that $V$ is solid.
\end{proof}


\begin{lem}\label{uncountable}
Let $H$ be a connected linear algebraic group defined over an infinite perfect field $K$, and $L$ a $K$-closed proper subgroup. Then $H(K)/L(K)$ is has the same cardinality as $K$.
\end{lem}
\begin{proof}
(Beware that the canonical injective map $H(K)/L(K)\to (H/L)(K)$ can fail to be surjective, so it is not enough to compute the cardinal of the latter.)

Clearly the cardinal of $H(K)$ is bounded above by that of $K$ (as soon as $K$ is infinite), so we have to prove the other inequality.

We first take for granted that there exists a $K$-closed curve $C$ in $H$, $K$-birational to $\mathbb{P}^1$, such that $C\nsubseteq L$ and $1\in C$. This granted, let us conclude (without the perfectness restriction on $K$). 

On $C$, we consider the equivalence relation $x\sim y$ if $x^{-1}y\in L$. This is a closed subvariety of $C\times C$, and does not contain any layer $C\times\{y_0\}$ since $y_0\in L$ and then $C\subset L$ would follow. So equivalence classes are finite. Since $C$ is $K$-birational to $\mathbb{P}^1$, its cardinal is the same as $K$, and then its image in the quotient $H(K)/L(K)$ being the quotient of $C(K)$ by the equivalence relation $\sim$ with finite classes, it also has at least the cardinal of $K$.

To justify the existence of $C$, we can argue that $H$ is $K$-unirational (this uses that $K$ is perfect), as established in \cite[Corollary 7.12]{BoS}; consider a dominant $K$-defined morphism $f:U\to H$, with $U$ open in the affine $d$-space $\mathbb{A}^d$. Conjugating with translations both in $\mathbb{A}^d$ and $H$, we can suppose that $0\in U$ and $f(0)=1$. Since $U(K)$ is Zariski-dense and $f$ is dominant, $f(U)$ is Zariski-dense, and hence it contains a point $f(x)\notin L$ for some $x\in U$. Let $D\subset\mathbb{A}^d$ be the line through $x$. Then the Zariski closure of $f(D\cap U)$ is the desired curve.
\end{proof}

\begin{cor}\label{solsubide}
In a fd Lie algebra over a field of characteristic 0, every solid subspace is an ideal.
\end{cor}
\begin{proof}
We have to show that if $V$ is $\Aut(\g)_0$-invariant then it is an ideal. Indeed, differentiation implies that $V$ is invariant under derivations of $\g$, and this includes inner derivations. Since left and right multiplications are derivations, this finishes the proof.
\end{proof}

\begin{ex}[A solid complete flag in a filiform Lie algebra of dimension $\ge 4$] \label{flagfili}
Recall that, for $d\ge 2$, a filiform Lie algebra denotes a $d$-dimensional nilpotent Lie algebra whose nilpotency length is exactly $d-1$, which is the largest possible value. Such a Lie algebra $\g$ admits a basis $(e_1,\dots,e_d)$ such that, denoting by $\g_{\ge i}$ the subspace with basis $(e_i,\dots,e_d)$, we have, for all $i\in\{2,\dots,d\}$, $\g^i=\g_{\ge i+1}$.  In addition, if $d\ge 4$, we can arrange to choose $[e_1,e_2]=e_3$ and $[e_1,e_3]=e_4$. This being assumed, each $\g_{\ge i}$ is a solid ideal. Indeed, for $i\neq 2$, this is because it is a term of the lower central series. For $i=2$, this is because it is the centralizer of $\g^2(=\g^{\ge 3})$ modulo $\g^4(=\g_{\ge 5})$. 
\end{ex}

\section{Systolic growths: facts and bounds}

\subsection{Asymptotic comparison}\label{s_asy}
Given functions $f,g$ (of a positive real variable $r$), we write $f\preceq g$ if $f(r)\le Cg(C'r)+C''$ for some constants $C,C',C''>0$ and all $r$. We say that $f$ and $g$ are $\simeq$-equivalent and write $f\simeq g$ if $f\preceq g\preceq f$. Also, we write $f\ll g$ if $g/(|f|+1)\to +\infty$ (usually $\underline{\lim}_{r\to\infty}f>0$, in which case this just means $g/f\to +\infty$).

\subsection{Residual girth}\label{s_regi}

The residual girth is defined in the same way as the systolic growth, but allowing only normal subgroups.

Let $G$ be a simply connected nilpotent Lie group, of dimension $d$ and nilpotency length $c$ with a lattice $\Gamma$. We can find an embedding of $G$ into the group of upper triangular unipotent real matrices, mapping $\Gamma$ into integral matrices. Then congruence subgroups (the kernel of reduction modulo $n$, in restriction to $\Gamma$) have index $\simeq n^d$ and systole $\simeq n^{1/c}$. This yields for $\Gamma$ the polynomial upper bound on the residual girth $\preceq r\mapsto r^{cd}$; this simple observation was made independently in \cite{BS,gralie}. As observed in \cite{gralie}, in the case of the 3-dimensional Heisenberg group, it is sharp: the residual girth of every lattice is $\simeq r^6$. In \cite{BS}, it is shown that, more generally, the residual girth of every lattice is indeed $\simeq r^{cd}$ when the center of $G$ coincides with the $c$-th term of the central series. Otherwise the picture is not completely clear. 

The above provide an easy polynomial upper bound on the systolic growth of $\Gamma$ and of $G$ are also $\preceq r^{cd}$. Nevertheless, we will not concentrate further on the residual girth $\sigma_\Gamma^\lhd$, inasmuch as its asymptotic behavior is usually much larger than the systolic growth $\sigma_\Gamma$: as soon as $G$ is non-abelian, $\sigma_\Gamma(r)\ll r^{cd}$, and it is also very likely that $\sigma_\Gamma\ll\sigma_\Gamma^\lhd$ always holds.

%

\subsection{Systolic growths}\label{s_usg}
Another notion, also related with conjugacy phenomena, but more closely related to the systolic growth, is the uniform systolic growth introduced in \cite{gralie}: we define it now.

Let $H$ be an arbitrary compactly generated locally compact group. We use the standard natural convention $\inf\emptyset=+\infty$. In the setting we will study more deeply (compactly generated locally compact nilpotent groups), all lattices are cocompact. So we stick to cocompact lattices, but in general it would make sense to consider the analogous notion allowing arbitrary lattices.

In the setting of Definition \ref{d_sy}, one can define the uniform systole, or $H$-uniform systole, of $X\subset H$ as the infimum
\[\inf_{h\in H}\mathrm{sys}(hXh^{-1})=\inf\{|hxh^{-1}|:\;h\in H,x\in X\smallsetminus\{1\}.\]

The uniform (or $H$-uniform) systolic growth of $H$ is then defined as the function mapping $r>0$ to the infimum $\sigma^u(r)\in\mathbf{R}_+\cup\{+\infty\}$ of covolumes of cocompact lattices of $H$ with uniform systole $\ge r$.

%
%
%
%

Given a cocompact lattice $\Gamma$ in $H$, we can consider the uniform systolic growth $\sigma_\Gamma$ of $\Gamma$ (computed within $\Gamma$) and the uniform systolic growth $\sigma_H$ of $H$. But while $\sigma_H\preceq\sigma_\Gamma$, a similar estimate for the uniform systolic growths might fail because for a finite index subgroup of $\Gamma$, the $H$-uniform systole can be much smaller than the $\Gamma$-uniform systole (see Example \ref{e_unisy}, in the Heisenberg group). Hence another notion naturally appears:
the $H$-uniform systolic growth $\sigma^u_{\Gamma,H}$ of $\Gamma$, considering the uniform systole computed in $H$ but finite index subgroups of $\Gamma$. At this point we have a bunch of growths, which we summarize now (up to asymptotic equivalence, allowing to not specify the choice of lengths): each maps $r\ge 0$ to the smallest covolume of a cocompact lattice of $H$ with the additional conditions:
\begin{itemize}
\item (systolic growth $\sigma_H$ of $H$): of systole $\ge r$;
\item (systolic growth $\sigma_\Gamma$ of $\Gamma$: contained in $\Gamma$, of systole $\ge r$;
\item (uniform systolic growth $\sigma^u_H$ of $H$): of $H$-uniform systole $\ge r$
\item (uniform systolic growth $\sigma^u_\Gamma$ of $\Gamma$): contained in $\Gamma$, of $\Gamma$-uniform systole $\ge r$
\item ($H$-uniform systolic growth $\sigma^u_{\Gamma,H}$ of $\Gamma$): contained in $\Gamma$, of $H$-uniform systole $\ge r$.
\end{itemize}

with asymptotic inequalities
\[\sigma_H\preceq\sigma_\Gamma\preceq\sigma_\Gamma^u\preceq\sigma_{\Gamma,H}^u;\quad   \sigma_H\preceq\sigma_H^u\preceq\sigma^u_{\Gamma,H}.\]

(Let us also mention that all these functions are $\preceq\sigma^\lhd_\Gamma$: the only nontrivial case is that of $\sigma^u_{\Gamma,H}$, and follows from the fact that for a subgroup $\Lambda$ normalized by a fixed cocompact lattice $\Gamma$, the $H$-normal systole is bounded by the $\Gamma$-uniform systole plus a constant depending only on $\Gamma$, independently of $\Lambda$.)

In the context when $H=G$ is a simply connected nilpotent Lie group, all these functions take finite values and are asymptotically bounded above by the residual girth of $\Gamma$, and in particular are polynomially bounded by $r\mapsto r^{c\dim(G)}$, a better upper bound (implying $\preceq r^{c\dim(G)/2}$ is provided in Proposition \ref{ineq}. In this context, we do not know if all these five functions have the same asymptotic behavior (Question \ref{q_GamG}). In various cases  where we prove an upper bound on the systolic growth, by constructing an explicit sequence of lattices, we actually provide an upper bound on $\sigma_{\Gamma,G}^u$ and therefore on all others.

\begin{rem}\label{riemannian}Most of these functions can be interpreted in the geometry of $G/\Gamma$. More precisely, endow $G$ with a right-invariant Riemannian metric, which thus passes to the quotient $G/\Gamma$, as well as its covering $G/\Lambda$ when $\Lambda$ is a finite index subgroup of $\Gamma$. Then $\Gamma$ can naturally be identified to $\pi_1(G/\Gamma)$ by a bijection $\gamma\mapsto j_\gamma$ (the base-point is meant to be the obvious one). The length of $\gamma\in\Gamma$ is equivalent to the length in $X=G/\Gamma$ of a smallest representative based loop of $j_\gamma$. Its $G$-uniform length is equivalent to the infimum of lengths of arbitrary loops in the free homotopy class of $j_\gamma$. Therefore, the systole (resp.\ $G$-uniform systole) of $\Lambda$ is equivalent  the smallest size of a based loop (resp.\ of a loop) in $X$ not homotopic to a point. Call the latter the geometric based systole, resp.\ geometric systole, of $X$ (classically, the word ``geometric" is dropped, since these notions come from Riemannian geometry!). The geometric based systolic growth, resp.\ geometric systolic growth, of $X$ is defined as the function mapping $r$ to the smallest degree of a covering of $X$ with geometric based systole, resp.\ systole $\ge r$. Thus the geometric based systolic growth of $X$ coincides with the systolic growth of $\Gamma$, and the geometric systolic growth of $X$ coincides with the $G$-uniform systolic growth of $\Gamma$.
\end{rem}

\begin{ex}\label{e_unisy}
In the 3-dimensional real Heisenberg group $G$, let us write, for the sake of shortness, $M(a,b,c)=\begin{pmatrix}1 & a & c\\ 0 & 1 & b\\ 0 & 0 & 1\end{pmatrix}$. We write $|M(a,b,c)|=|a|+|b|+\sqrt{|c|}$; we use this approximation of the distance to the origin to compute systoles.

Let $\Gamma_n$ be the subgroup generated by the matrices $x_n=M(1,0,n)$, $y_n=M(0,n^2,0)$; this is a lattice. We claim that its $G$-uniform systole is 1 while its $\Gamma_n$-uniform systole is $\ge\sqrt{n}$.

Let us describe $\Gamma_n$. Define $z_n=x_ny_nx_n^{-1}y_n^{-1}$. Then $z_n=M(0,0,n^2)$ and elements of $\Gamma_n$ are precisely those $x_n^ay_n^bz_n^c$ when $(a,b,c)$ ranges over $\mathbf{Z}^3$. We see that \[x_n^ay_n^bz_n^c=M(a,bn^2,n(a+(ab+c)n)).\] 

In $G$, $x_n$ is conjugate to $M(1,0,t)$ for every real $t$, and hence the $G$-systole of $\Gamma_n$ is equal to 1. 

Let us compute the $\Gamma_n$-uniform systole. Consider $(a,b,c)\in\mathbf{Z}^3\smallsetminus\{0\}$ and consider $g=x_n^ay_n^bz_n^c$, and $g'$ any conjugate of $g$ (for the moment, a $G$-conjugate). 
\begin{itemize}
\item If $(a,b)=(0,0)$, then $c\neq 0$ and $g=M(0,0,cn^2)$ is central, so $|g'|=|g|=\sqrt{|c|}n\ge n$;
\item if $b\neq 0$, then $|g'|\ge |a|+|b|n^2\ge n^2$; 
\item if $b=0$ and $a\neq 0$, then $g=x_n^az_n^c=M(a,0,n(a+cn))$. We discuss
	\begin{itemize}
	\item if $|a|\ge\sqrt{n}$, then $|g'|\ge\sqrt{n}$;
	\item if $|a|<\sqrt{n}$, we now assume that $g'$ is a $\Gamma_n$-conjugate of $g$, namely by some element $M(*,dn^2,*')$ with $d\in\mathbf{Z}$ (we do not have to care about the coefficients denoted by stars); this gives \[g'=M(a,0,n(a-(ad-c)n)).\] If by contradiction $|g'|<\sqrt{n}$, then $a-(ad-c)n=0$, but since $|a|<n$ this implies $a=0$, a contradiction. So $|g'|\ge\sqrt{n}$.
	\end{itemize}
\end{itemize}
We conclude that in all cases $|g'|\ge\sqrt{n}$, so the $\Gamma_n$-systole of $\Gamma_n$ is $\ge\sqrt{n}$.
\end{ex}

\section{Algebraization of the systolic growth}\label{secsyg}

The purpose of this section is to describe the various systolic notions in terms of discrete cocompact subrings of the Lie algebra, instead of lattices in the Lie group. While the exponential of a discrete cocompact subring can fail to be a subgroup and the logarithm of a lattice can fail to be an additive subgroup or fail to be stable under taking brackets, the correspondence is true up to bounded index. This is the contents of the following lemma.

\begin{lem}[Folklore]\label{llazard}
Let $\g$ be a real fd nilpotent Lie algebra and $G$ the corresponding simply connected nilpotent Lie group. There exists $C\ge 1$ depending only on $\dim(\g)$ such that:
\begin{enumerate}
\item\label{llaz1} for every cocompact discrete subring $\Lambda$ in $\g$, there exist lattices $\Gamma_1,\Gamma_2$ in $G$ with $\Gamma_1\subset\exp(\Lambda)\subset\Gamma_2$ and $[\Gamma_2:\Gamma_1]\le C$.
\item\label{llaz2} for every lattice $\Gamma$ in $G$, there exist cocompact discrete subrings $\Lambda_1,\Lambda_2$ in $\g$ with $\Lambda_1\subset\log(\Gamma)\subset\Lambda_2$ and $[\Lambda_2:\Lambda_1]\le C$.
\end{enumerate}
\end{lem}

We include a proof (of a slightly stronger statement) in Appendix \ref{app_nil}.

\subsection{Systolic growth of a real nilpotent Lie algebra}

We define a Lie algebra analogue of the systole and systolic growth as follows (fix some Lebesgue measure on the vector space $\g$). Recall that $\lf\cdot\rf$ denotes the Guivarch length, introduced in \S\ref{pre_la}.

\begin{defn}
If $H\subset\g$, define 
\[\sys(H)=\inf\{\lf h\rf:h\in H\smallsetminus\{0\}\}\in\mathbf{R}_+\cup\{+\infty\}.\] Also define the uniform (or $G$-uniform) systole as $\sys^u(H)=\inf_{g\in G}g\cdot\sys(H)$ (where $G$ acts through the adjoint representation).

The systolic (resp.\ uniform systolic) growth of $(\g,\lf\cdot\rf)$ is the function mapping $r\ge 0$ to the infimum $\sigma(r)\in\mathbf{R}_+\cup\{+\infty\}$ of covolumes of cocompact discrete subrings of $\g$ with systole (resp.\ uniform systole) $\ge r$.
\end{defn}

Note that the systolic growth also depends on the choice of normalization of the Lebesgue measure. Its asymptotic growth, however, only depends on the real Lie algebra $\g$. Its interest lies in the following fact:

\begin{prop}\label{lalgsg}
The systolic (resp.\ uniform systolic) growth of $\g$ (as a Lie algebra) and the systolic (resp.\ uniform systolic) growth of $G$ are $\simeq$-equivalent. 
\end{prop}
\begin{proof}
The exponential map preserves the systole of subsets up to a bounded multiplicative error, by Guivarch's estimates. So the proposition would be already proved if the exponential map were giving an exact correspondence between cocompact discrete subrings of $\g$ and lattices in $G$. This is not the case, but is, however, ``true up to a bounded index error", as explained in Lemma \ref{llazard}, which entails the result for systolic growth.

The uniform case follows, using in addition that the exponential map $\g\to G$ is $G$-equivariant, for the adjoint action on $\g$ and the conjugation action on $G$.
\end{proof}

\subsection{Systolic growth of a rational nilpotent Lie algebra}\label{syrnla}

To estimate the systolic growth in the lattice, we need a similar notion pertaining to rational Lie algebras. Namely, let $\mk{l}$ be a fd nilpotent Lie algebra over $\mathbf{Q}$. Define its ``realification" $\g=\mk{l}\otimes_\mathbf{Q}\mathbf{R}$. Choose $\mk{v}_i$ and norms as above on $\g$, so as to define $\lf\cdot\rf$; so we have a notion of systole for subsets $\g$, and choose a Lebesgue measure on $\g$.

\begin{defn}
The systolic (resp.\ $G$-uniform systolic) growth of the rational Lie algebra $(\mk{l},\lf\cdot\rf)$ is the function mapping $r\ge 0$ to the infimum $\sigma(r)\in\mathbf{R}_+\cup\{+\infty\}$ of covolumes of cocompact discrete subrings of $\g$ contained in $\mk{l}$, with systole (resp.\ $G$-uniform systole) $\ge r$. 
\end{defn}

Note that the only difference in the last definition is that we restrict to those subrings contained in $\mk{l}$. In particular, this function is bounded below by the systolic growth of $G$ (relative to the same choice of norms, etc.). Also note that we did not attempt to define an analogue of the $\Gamma$-uniform systolic growth.

Denoting $L=\exp(\mk{l})$ (which can be thought as the group $G_\mathbf{Q}$ of $\mathbf{Q}$-points of $G$ for a suitable rational structure), we have

\begin{prop}\label{equisystra}
If $\Gamma$ is any lattice in $G$ contained in $L$, then the systolic (resp.\ $G$-uniform systolic) growth of $\Gamma$ is $\simeq$-equivalent to the systolic (resp.\ $G$-uniform systolic) growth of the rational Lie algebra $\mk{l}$.
\end{prop}
\begin{proof}
An observation is that in Lemma \ref{llazard}, if $\Lambda\subset\mk{l}$ then automatically $\Gamma_i\subset L$, and in the other direction if $\Gamma\subset L$ then automatically $\Lambda_i\subset\mk{l}$. This being granted, the proof follows the same lines as that of Proposition \ref{lalgsg}.
\end{proof}

\section{General upper bounds on the systolic growth}
We prove here Proposition \ref{ineq}, giving here a more general result since we consider the uniform systolic growth. 

\begin{prop}\label{ineq}
Let $G$ be a simply connected nilpotent Lie group with a lattice $\Gamma$. Let $D$ be its homogeneous dimension and let $k$ be defined in \S\ref{s_iub}. Then $\sigma_{\Gamma,G}^u(r)\preceq r^{D+k}$, and hence the systolic growth and $G$-uniform systolic growth of both $G$ and $\Gamma$ are all $\preceq r^{D+k}$.
\end{prop}
\begin{proof}
It is enough to show that the $G$-uniform systolic growth of $\Gamma$ is $\preceq r^{k'}$, where 
\[k'=\frac{c}2\dim(\g/\g^{\lceil c/2\rceil})+\sum_{i=\lceil c/2\rceil}^c i\dim(\g^i/\g^{i+1})=\sum_{1}^c \max(c/2,i)\dim(\g^i/\g^{i+1}).\]
Let $G_\mathbf{Q}\subset G$ be the rational Malcev closure of $\Gamma$ and $\g_\mathbf{Q}$ its Lie algebra. From the lower central series we choose supplement subspace to obtain a vector space decomposition $\g=\bigoplus_{i=1}^c\g_i$, with $\bigoplus_{j\ge i}\g_j=\g^i$ for all $i$ (this is not necessarily a Lie algebra grading). We choose bases of all this subspaces so as to ensure that all structure constants are integral; thus $\g_i(\mathbf{Z})$ means the discrete subgroup generated by the given basis of $\g_i$.

We now define, for every square integer $n$, $\Lambda_n=\bigoplus_{i=1}^cn^{\max(c/2,i)}\g_i(\mathbf{Z})$. 
We have to check that this is a subalgebra, namely that \[B_{ij}:=[n^{\max(c/2,i)}\g_i(\mathbf{Z}),n^{\max(c/2,j)}\g_j(\mathbf{Z})]\subset\Lambda_n\] for all $i,j$. Indeed, keeping in mind that $n$ is a square,
\[B_{ij}\subset n^c[\g_i(\mathbf{Z}),\g_j(\mathbf{Z})]\subset n^c\bigoplus_{p=i+j}^c\g_p(\mathbf{Z})\subset \Lambda_n.\]
The $G$-uniform systole of $\Lambda_n$ is $\succeq n$: indeed, any nonzero element has the form $w=\sum_{j\ge i}n^jv_j$ with $v_i\in\g_i(\mathbf{Z})\smallsetminus\{0\}$. So any $G$-conjugate of $w$ has the form $n^iv_i+\mu$ with $\mu\in\g^{i+1}$ and hence has norm $\ge n^i$ (for the $\ell^1$-norm with respect to the fixed basis), and thus has Guivarch length $\ge n$. Then $\Lambda_n$ precisely has covolume $n^{k'}$.

Using that the exponential map $\g\to G$ is $G$-equivariant and in view of Lemma \ref{llazard}, we deduce the desired upper bound.
\end{proof}

The following proposition provides upper bounds on $k$ and $k'=k+D$ depending only on the dimension $d$. Recall that the maximal homogeneous dimension $D$ for given dimension $d\ge 2$ is equal to $\frac{d(d-1)}2+1$ and is precisely attained for filiform Lie algebras, which are those of maximal nilpotency class (namely $d-1$).  

\begin{prop}\label{ub_kc}
Let $\g$ be a finite-dimensional nilpotent Lie algebra of dimension $d$,  nilpotency length $c$, and homogeneous dimension $D$, and $k=k_c(\g)$ is defined in \S\ref{s_iub}. Then \[k\le \frac{d^2}6 -\frac{d}2+\frac{1}{2}\quad\text{and}\quad k+D\le \frac{5d^2-4d}8.\]
\end{prop}
\begin{proof}
We denote $b=\lceil c/2\rceil$ and $d=\dim(\g)$. Then
\begin{align*}
k= & \sum_{i=1}^{b-1}\left(\frac{c}2-i\right)\dim(\g^i/\g^{i+1}) \\
	=& \frac{c}2\dim(\g/\g^b)-\dim(\g/\g^2)-\sum_{i=2}^{b-1}i\dim(\g^i/\g^{i+1});\\
= & \frac{c}2(d-\dim(\g^b))-(d-\dim(\g^2))-\sum_{i=2}^{b-1}i\dim(\g^i/\g^{i+1});
\end{align*}
for $i\ge 2$ we use the inequality $\dim(\g^i/\g^{i+1})\ge 1$, $\dim(\g^i)\ge c-i+1$ for $1\le i\le c$ (applied for $i=2$ and $i=b$) and get
\begin{align*}
k\le & \frac{c}2(d-c+b-1)-(d-c+1)-\frac{b(b-1)}2+1;
\end{align*}
write $c=2b-e$ with $e\in\{0,1\}$, this yields
\[k\le\frac{-3c^2+2(2d+3)c-8d+e}8.\]

A polynomial function of the form $-3x^2+2ax$ is maximal for $x=a/3$, where it takes the value $a^2/3$. Hence 
\[k\le\frac{\frac13(2d+3)^2-8d+e}8=\frac16d^2-\frac12d+\frac{3+e}8.\] 

Now consider $k'=k+D$ and again write $c=2b-e$ with $e\in\{0,1\}$. Then
\begin{align*}
k'= & \frac{c}2\dim\g+\sum_{i=b}^{c}\left(i-\frac{c}2\right)\dim(\g^i/\g^{i+1}) \\
	=& \frac{cd}2+\sum_{i=b}^{c}\left(i-\frac{c}2\right)+\sum_{i=b}^{c}\left(i-\frac{c}2\right)(\dim(\g^i/\g^{i+1})-1);\\
	\le & \frac{cd}2+\sum_{i=b}^{c}\left(i-\frac{c}2\right)+\frac{c}2\sum_{i=b}^{c}(\dim(\g^i/\g^{i+1})-1);\\
	 & = \frac{cd}2+\sum_{i=b}^{c}(i-c)+\frac{c}2\sum_{i=b}^{c}\dim(\g^i/\g^{i+1});\\
 &= \frac{cd}2+\frac{c(c+1)}{2}-\frac{b(b-1)}{2}-c(c-b+1)+\frac{c}2\sum_{i=b}^{c}(\dim(\g^i/\g^{i+1})-1);\\
 &= \frac{4cd-c^2+(2e-2)c+e}{8}+\frac{c}2\dim(\g^b).\\
\end{align*}

For all $i\ge 2$ we have $\dim\g^i\le d-i$. Therefore we have (assuming $c\ge 3$, so $b\ge 2$)
\[k'\le  \frac{4cd-c^2+(2e-2)c+e}{8}+\frac{c}2(d-b)=\frac{8cd-3c^2-2c+e}{8}.\]

A function of the form $x\mapsto -3x^2+ax$ increases until it takes its maximum for $x=a/6$; actually, using $d\ge 4$, $d-1\le (8d-7e-6)/6$, and since $c\le d-1$, so the last expression is bounded above by the same when $c$ is replaced with its maximum possible value $d-1$. Hence
\[k'\le  \frac{8(d-1)d-3(d-1)^2-2(d-1)+e}{8}=\frac{5d^2-4d+e-1}{8}.\qedhere\]
\end{proof}

\section{The strategy for lower bounds}\label{s_thestra}

We consider a fd $c$-step nilpotent real Lie algebra $\g$, with lower central series $(\g^i)_{i\ge 1}$ and dimension $d$. We decompose it as a direct sum of subspaces $\g=\bigoplus_{i=1}^c\mk{v}_i$, so that for all $i$, $\g^i=\bigoplus_{j\ge i}\mk{v}_j$ (we call this a compatible decomposition).

Let us choose a basis of $\g$ as a concatenation of bases of the $\mk{v}_i$. Note that the basis determines the subspaces $\mk{v}_i$, which is spanned by the subset of basis elements that belong to $\g^i\smallsetminus\g^{i+1}$. 

It is convenient to also directly define such a basis (without defining the $\mk{v}_i$ beforehand). A basis $(e_1,\dots,e_d)$ in a fd nilpotent Lie algebra is compatible if it satisfies the following three conditions:
\begin{itemize}
\item for all $i,j$, $e_j\in\g^i\smallsetminus\g^{i+1}$ and $e_k\in\g^{i+1}$ implies $j<k$;
\item $\{e_j:j\ge 1\}\cap\g^i$ spans $\g^i$ for all $i$;
\item the subspaces $\g_i=\langle e_j:j\ge 1,e_j\in\g^i\smallsetminus\g^{i+1}\rangle$ span their direct sum.
\end{itemize}

Thus any compatible basis determines a compatible decomposition, and any compatible decomposition yields compatible bases.

In the real case, it is convenient to assume that the nonzero structure constants of $\g$ with respect to this basis have absolute value $\ge 1$ (this is a mild assumption as it always hold after replacing the basis by a scalar multiple), thus not changing the compatible decomposition.

We define a compatible flag as a sequence of ideals 
\begin{equation}\g=\mk{w}_1>\mk{w}_2>\dots>\mk{w}_k=\{0\},\label{flag}\end{equation}
such that each term $\g^i$ occurs among the $\mk{w}_j$. 

\subsection{Dilations}\label{s_dil}
Fix a nilpotent fd Lie algebra with a compatible decomposition $\g=\bigoplus_{i=1}^c\mk{v}_i$.

For a nonzero scalar $r$ define the diagonal linear automorphism $u(r)$ of $\g$ given by multiplication by $r^i$ on $\mk{v}_i$.

We define a new Lie algebra $\g[r]$, with underlying $k$-linear space $\g$, with the bracket $[x,y]_r=u(r)^{-1}[u(r)x,u(r)y]$, that is, the pull-back of the original bracket by the linear automorphism $u(r)$.

\begin{ex}
Consider the 5-dimensional Lie algebra with basis $(e_1,\dots,e_5)$ and nonzero brackets $[e_1,e_3]=e_4$, $[e_1,e_4]=[e_2,e_3]=e_5$. This is a compatible basis, with $\mk{v}_1,\mk{v}_2,\mk{v}_3$ having bases $(e_1,e_2,e_3)$, $(e_4)$ and $(e_5)$ respectively. Then $\g[r]$ has, in this basis, the nonzero brackets $[e_1,e_3]_r=e_4$, $[e_1,e_4]_r=e_5$, and $[e_2,e_3]_r=r^{-1}e_5$.
\end{ex}

\begin{rem}
If $\g=\bigoplus\mk{v}_i$ is a Lie algebra grading (and thus a Carnot grading), then $[\cdot,\cdot]_r=[\cdot,\cdot]$. Beware that for an arbitrary  (e.g., Carnot) nilpotent Lie algebra of nilpotency length $c\ge 3$, we can find a compatible decomposition $\g=\bigoplus\mk{v}_i$ such that $[\mk{v}_1,\mk{v}_1]$ is not contained in $\mk{v}_2$.
\end{rem}

\begin{rem}
We have $\g=\g[1]$, and the function $(g,h,r)\mapsto [g,h]_r$ is polynomial with respect to $r^{-1}$ (with each coefficient a skew-symmetric bilinear map $\g\times\g\to\g$). The constant coefficient $[\cdot,\cdot]_\infty$ of this polynomial can be thought as the limit of the brackets $[\cdot,\cdot]_r$ when $r\to\infty$ (this is indeed the case when $K$ is a normed field), and 
$(\g,[\cdot,\cdot]_\infty)$ is the associated Carnot-graded Lie algebra of $\g$, in which for $x\in\mk{v}_i$ and $y\in\mk{v}_j$, the bracket $[x,y]_\infty$ is defined as the projection of $[x,y]$ on $\mk{v}_{i+j}$ modulo $\g^{i+j+1}$. This 1-parameter family of brackets has been used by Pansu and Breuillard \cite{Pan,Bre}.
\end{rem}

\subsection{Renormalization of the algebraic systolic growth}

As in \S\ref{s_dil}, we fix a nilpotent fd Lie algebra with a compatible decomposition $\g=\bigoplus_{i=1}^c\mk{v}_i$, and assume in addition that the ground field is the field of real numbers. Then the dilation $u(r)$ multiplies the Lebesgue measure (for any choice of normalization) by $|r|^D$, where $D$ is the homogeneous dimension of $\g$ (see \S\ref{pre_la}).

We fix a norm on $\g$ which is the sup-norm with respect to this decomposition, that is, for every $x=\sum x_i$, $x_i\in\mk{v}_i$, we have $\|x\|=\max\|x_i\|$. This implies that $u(r)$ multiplies the Guivarch length by $|r|$. This also implies that for $x\in\g$, the conditions $\|x\|\ge 1$ and $\lfloor x\rfloor\ge 1$ are equivalent.

Let $\Lambda_r$ be a discrete cocompact subring of Guivarch systole $\ge r$. For instance, assuming that $\g$ is admits rational structure, we can choose, by compactness, $\Lambda_r$ to have covolume exactly $\sigma(r)$ (i.e., has minimal covolume among those discrete cocompact subrings of Guivarch systole $\ge r$.

Then $\Xi_r=u(r)^{-1}\Lambda_r$ has Guivarch systole $\ge 1$; this is a discrete cocompact subring in $\g[r]$. We have $\cov(\Xi_r)=r^{-D}\cov(\Lambda_r)$.
 
The idea is that rather than studying $\Lambda_r$ (whose systole tends to infinity), we study $\Xi_r$ (which has bounded systole $\ge 1$), the only caveat being that $\Xi_r$ is a Lie subring of $\g[r]$ (i.e., for some Lie bracket which varies with $r$). 
 
\subsection{First application: small systolic growth}\label{s_small}

Assume here that $\underline{\lim}r^{-D}\sigma(r)<\infty$ and let us prove that $\g$ is Carnot.

Suppose that $\cov(\Lambda_r)\le Cr^D$ (for $r\in I$, where $I$ is an unbounded set of positive real numbers). Then $\cov(\Xi_r)\le C$. By compactness of the set of lattices with systole $\ge 1$ and covolume $\le C$, we can find an unbounded subset $J\subset I$ such that $\lim_{J\ni r\to\infty}\Xi_r=\Xi_\infty$ (in the Chabauty topology) for some lattice $\Xi_\infty$. Clearly $\Xi_\infty$ is a subring of $(\g,[\cdot,\cdot]_\infty)$. 

Then $\Xi_r$ is isomorphic as a Lie ring to $\Xi_\infty$ for large enough $r\in J$. Indeed, choose a basis $(u_1,\dots,u_d)$ of $\Xi_\infty$ as a $\mathbf{Z}$-module. Then we can find $u_i^r\in\Xi_r$ with $u_i^r\to u_i$. Then $[u_i,u_j]=\sum_kn_{ij}^ku_k$ for suitable integers $n_{ij}^k$ (here exponents are additional indices, and not powers). Then $[u_i^r,u_j^r]-\sum_kn_{ij}^ku_k^r$ belongs to $\Xi_r$ and converges to zero, hence is zero for $r$ large enough. Moreover, since the covolume is a continuous function, eventually $(u_i^r)$ is a $\mathbf{Z}$-basis of $\Xi_r$. This shows that this is indeed an isomorphism. Since $\g\simeq\Xi_r\otimes_Z\mathbf{R}$ and $(\g,[\cdot,\cdot]_\infty)\simeq\Xi_\infty\otimes_\mathbf{Z}\mathbf{R}$ as real Lie algebras, this shows that $(\g,[\cdot,\cdot]_\infty)\simeq\g$. Thus $\g$ is Carnot.

\subsection{Covolume inequalities}\label{covine}

(In this \S\ref{covine}, the Lie algebra bracket plays no role.) 
Assume now that the ground field is the field of real numbers. 

We fix a compatible flag as in (\ref{flag}), denoted $F$ for short. 
We fix Lebesgue measures on all $\mk{w}_j/\mk{w}_{j-1}$, and thus on $\g$.
Fix a compatible basis, compatible with this flag.
This basis defines (in a compatible way) normalizations of the Lebesgue measures on each quotient $\mk{w}_j/\mk{w}_k$, where the cube $[0,1\mathclose[^q$ (for the given basis) has volume 1, and sup norms with respect to this basis. 

The basis also yields a compatible decomposition; we thus have a notion of Guivarch length (\S\ref{pre_la}). Observe that $u(r)$ multiplies the Guivarch length by $|r|$.

We say that an additive lattice $\Lambda$ of $(\g,+)$ is $F$-compatible if $\Lambda\cap\mk{w}_j$ is a lattice in $\mk{w}_j$ for every $j$.
This implies that for all $j$ the projection $\Lambda_{[j]}$ of $\Lambda\cap\mk{w}_j$ on $\mk{f}_j=\mk{w}_j/\mk{w}_{j-1}$ is a lattice in $\mk{f}_j$ as well.
Let $a_j(\Lambda)$ be the covolume of $\Lambda_{[j]}$ in $\mk{f}_j$. Then the covolume of $\Lambda\cap\mk{w}_j$ is equal to $\prod_{k\ge j}a_k(\Lambda)$.

Note that for any $v\in\g$, we have $\|v\|\ge 1$ if and only if $\lfloor v\rfloor\ge 1$. So for an additive lattice of systole $\ge 1$ (for either the norm or the Guivarch length), the covolume is $\ge 1$. Thus, $\prod_{k\ge j}a_k(\Lambda)\ge 1$ for all $j$. 
This means that the following holds:

\begin{lem}\label{bas_ineq}
For any $j$, any $F$-compatible additive lattice of $\g$ of Guivarch systole $\ge 1$ satisfies 
\[\prod_{k\ge j}a_k(\Lambda)\ge 1,\quad\text{or equivalently,}\quad \sum_{k\ge j}\log_r(a_k(\Lambda))\ge 0,\quad\forall r>1.\qquad\qed\]
\end{lem}

In the sequel, these inequalities will be combined with other inequalities making use of further assumptions (namely that $\Lambda$ is a Lie subring).

The requirements about the norm will be fulfilled when we consider a Lie algebra with basis $(e_1,\dots,e_d)$, and a flag containing all elements of the lower central series such that, denoting $\g_{\ge i}$, each element of the flag is one of the $\g_{\ge i}$; on each $\g_{\ge i}/\g_{\ge j}$ the norm being the $\ell^\infty$ norm and the Lebesgue measure being normalized so that the cube $[0,1\mathclose[^k$ has measure 1.

\section{Precise estimates}

We give explicit estimates of the systolic growth for various illustrating examples. While we obtain upper bounds by easy explicit construction, we use the method of \S\ref{s_thestra} to obtain lower bounds.

\subsection{5-dimensional non-Carnot Lie algebras}\label{s_5dim}

Following a convenient custom, when we describe a Lie algebra by saying that the nonzero brackets are $[e_i,e_j]=f_{ij}$, we mean that $[e_j,e_i]=-f_{ij}$ and that all other brackets between basis elements are zero. That an algebra defined in such a way is indeed Lie is equivalent to say that
\[J(e_i,e_j,e_k)=[e_i,[e_j,e_k]]+[e_j,[e_k,e_i]]+[e_k,[e_i,e_j]]=0,\quad \forall i<j<k.\]

There are exactly two non-isomorphic non-Carnot nilpotent 5-dimensional real Lie algebras. Using the notation in \cite{graaf}, these are defined, in the basis $(e_1,\dots,e_5)$, by the nonzero brackets:
\begin{align*}
L_{5,5}:\;& \qquad [e_1,e_2]=e_4, && [e_1,e_4]=e_5, && [e_2,e_3]=e_5;\\
L_{5,6}:\;& [e_1,e_2]=e_3, \;[e_1,e_3]=e_4,&& [e_1,e_4]=e_5, && [e_2,e_3]=e_5.
\end{align*}
The lower central series in $L_{5,5}$ is $123/4/5$ (this concise notation means $\g^2=\g_{\ge 4}$ and $\g^3=\g_{\ge 5}$) and in $L_{5,6}$ it is $12/3/4/5$. Thus we can write the Lie algebra law of $\g[r]$ in each case:
\begin{align*}
L_{5,5}:\;& \qquad  [e_1,e_2]_r=e_4, && [e_1,e_4]_r=e_5, && [e_2,e_3]_r=r^{-1}e_5;\\
L_{5,6}:\;& [e_1,e_2]_r=e_3,\; [e_1,e_3]_r=e_4,&& [e_1,e_4]_r=e_5, && [e_2,e_3]_r=r^{-1}e_5.
\end{align*}

\begin{lem}\label{solid_5d}
In both cases, the complete flag $(\g_{\ge i})_{1\le i\le 5}$ is made up of solid ideals.
\end{lem}
\begin{proof}
All are part of the lower central series, except $\g_{\ge 2}$ in both cases and $\g_{\ge 3}$ for $L_{5,5}$. The upper central series is $12/34/5$ for $L_{5,5}$, so $\g_{\ge 3}$ is also solid in this case. Finally, $\g_{\ge 2}$ is the centralizer of $\g_{\ge 4}$ in both cases, so is solid. 
\end{proof}

Let now, in either case, $\Lambda_r$ be a discrete cocompact subring in $\g$, with Guivarch systole $\ge r$ and covolume $\sigma(r)$, and define $\Xi_r=u(r)^{-1}\Lambda_r$, which is a discrete cocompact subring of $\g[r]$ with systole $\ge 1$. As in \S\ref{covine}, let $a_i(r)$ be the systole of the projection of $\Xi_r\cap\g^i$ in $\g^i/\g^{i+1}$, and write $A_i=\log_r(a_i(r))$.

\begin{lem}\label{Aiine}
In both cases, we have, for all $r>0$, the inequalities $A_1+A_4\ge 0$, $A_2+A_3\ge 1$, $A_5\ge 0$. In particular, $\sum_{i=1}^5A_i\ge 1$.
\end{lem}
\begin{proof}It is convenient to denote by $o(i)$, resp.\ $O(i)$, an unspecified element of $\g_{\ge i+1}$, resp.\ $\g_{\ge i}$. By definition, $\Lambda$ contains elements $v_1=a_1e_1+o(1)$, $v_2=a_2e_2+o(2)$, $v_3=a_3e_3+o(3)$, $v_4=a_4e_4+o(4)$, $v_5=a_5e_5$.

Then $A_5\ge 0$ means $a_5=\|v_5\|\ge 1$, which is a trivial consequence of having systole $\ge 1$. 

Then $[o(1),O(4)]=[O(1),o(4)]=[o(2),O(3)]=[O(2),o(3)]=0$ in both cases. Since for both $L_{5,5}[r]$ and $L_{5,6}[r]$ we have $[e_1,e_4]_r=e_5$ and $[e_2,e_3]_r=r^{-1}e_5$, it follows that $[v_1,v_4]_r=a_1a_4e_5$ and $[v_2,v_3]_r=r^{-1}a_2a_3e_5$. Therefore $a_1a_4\ge 1$ and $r^{-1}a_2a_3\ge 1$, which means that $A_1+A_4\ge 0$ and $A_2+A_3\ge 1$. The last inequality follows
\[A_1+A_2+A_3+A_4+A_5=  (A_1+A_4)+(A_2+A_3)+A_5\ge 0+1+0=1.\]
\end{proof}

Lemma \ref{Aiine} shows that the covolume $\cov(\Lambda_r)=\prod_{i=1}a_i$ is $\ge r$, which in turn means that $\sigma(r)=\cov(\Lambda_r)=r^D\cov(\Xi_r)\ge r^{D+1}$. 

To get a reverse inequality, let us define $\Xi'_r$ as the additive lattice with basis $(e_1,re_2,e_3,e_4,e_5)$. This is indeed a subring of $(\g,[\cdot,\cdot]_r)$. Its covolume is $r$ and its systole is $\ge 1$. We thus define $\Lambda'_r=u(r)\Xi'_r$ (an explicit basis depends on $u(r)$, which is not the same for the two Lie algebras in consideration): it has Guivarch systole $r$ and covolume $r^{D+1}$. 

Accordingly, for both Lie algebras, for this choice of norm, compatible decomposition and covolume, $\sigma(r)=r^{D+1}$. 

Actually, the previous construction of lattices $(\Xi'_r)$ satisfies some additional features: first, when $n$ is a positive integer, $\Lambda'_n$ is contained in the subring $\g[\mathbf{Z}]$. This implies that the systolic growth of the corresponding lattices is $\preceq r^{D+1}$. Since real nilpotent Lie algebra up to dimension 5 have a unique rational structure up to automorphism (see \cite{graaf}), this implies that the systolic growth of the lattices is also $\simeq r^{D+1}$, proving \ref{t_5dim}.

Second, the $G$-uniform Guivarch systole of the lattices $\Xi'_r$ is also $\ge 1$: this is because every nonzero element has the form $ne_i+o(i)$ for some integer $n$ and hence all its $G$-conjugates still has the form $ne_i+o(i)$ and thus has Guivarch norm $\ge 1$. Therefore, this proves that the $G$-uniform systolic growth of $G$ and its lattice is $\simeq r^{D+1}$ as well.

\subsection{One example with non-integral polynomial degree}

Its law is given by the symbolic notation
\[ 12|3,\;13|4,\;14|5,\;15|6,\; 16|7,\]
\[23|5,\;24|6,\; 34|7.\]

This means that the nonzero brackets are given by $[e_1,e_2]=e_3$, $[e_1,e_3]=e_4$, etc. (It appears as $\g_{7,1,1(0)}$ in Magnin's classification \cite{Mag}.)

It is filiform, that is, its nilpotency length is as large as possible for the given dimension; its lower central series is given, in symbols, as $12/3/4/5/6/7$. All the $\g_{\ge i}$ are solid (see Example \ref{flagfili}).

Once more, let $\Lambda_r$ be a discrete cocompact subring of covolume $\le r$ and Guivarch systole $\sigma(r)$, $\Xi_r=u(r)^{-1}\Lambda_r$. We define $a_i,A_i$ as previously. Then, by Lemma \ref{bas_ineq},
\[\sum_{i=j}^7A_i\ge 0,\quad 1\le j\le 7\]
Also the law in $\Lambda_r$ is given by 
\[ 12|3,\;13|4,\;14|5,\;15|6,\; 16|7,\]
\[23|r^{-1}5,\;24|r^{-1}6,\; 34|r^{-1}7.\]

\begin{lem}
$2A_1+A_5+A_6\ge 0$ and $2(A_2+A_3+A_4)\ge 3$.
\end{lem}
\begin{proof}
There exist elements $v_i=a_ie_i+o(i)$ in $\Lambda_r$.

Then $[v_1,v_5]=a_1a_5e_6+o(6)$ and $[v_1,v_6]=a_1a_6e_7$. These two elements generate a lattice of covolume $a_1^2a_5a_6$ in $\g_{\ge 6}$. Since it has systole $\ge 1$, we deduce $a_1^2a_5a_6\ge 1$, that is, $2A_1+A_5+A_6\ge 0$;

Next $[v_2,v_3]=r^{-1}a_2a_3e_5+o(5)$, $[v_2,v_4]=r^{-1}a_2a_4e_6+o(6)$, $[v_3,v_4]=r^{-1}a_3a_4e_7$. These three elements generate a lattice of covolume $r^{-3}a_2^2a_3^3a_4^2$ in $\g_{\ge 5}$. Since it has systole $\ge 1$, we deduce $r^{-3}a_2^2a_3^3a_4^2\ge 1$, and hence $2(A_2+A_3+A_4)\ge 3$.
\end{proof}

\begin{cor}
$\sum A_i\ge 3/2$.
\end{cor}
\begin{proof}
Indeed, $A_5+A_6+A_7\ge 0$ by Lemma \ref{bas_ineq} and hence
\[2\sum A_i=2(A_2+A_3+A_4)+(2A_1+A_5+A_6)+(A_5+A_6+A_7)+A_7\ge 3.\qedhere\]
\end{proof}

We deduce $\sigma(r)\preceq n^{D+3/2}$ (here $D=2+\sum_2^7k=29$).

In the other direction, for $r\ge 1$ we define $\Xi'_r$ to be the lattice with basis $(e_1,\sqrt{r}e_2,\sqrt{r}e_3,\sqrt{r}e_4,e_5,e_6,e_7)$. This is a discrete cocompact subring in $\g[r]$, with Guivarch systole $1$ and covolume $r^{3/2}$. So $\Lambda'_r=u(r)M_r$ is a discrete cocompact subring of Guivarch systole $r$ and covolume $r^{D+3/2}$, so that $\sigma(r)\preceq r^{D+3/2}$. 
To conclude, $\sigma(r)\simeq r^{D+3/2}$ (with $D=29$).

Moreover, for $r$ a square integer, this lattice $\Lambda'_r$ has an integral basis, so this also provides an upper bound for the systolic growth of the lattices relative to the given rational structure.

\subsection{Truncated Witt Lie algebras}
Here we prove Theorem \ref{vira}. The method is similar to the approach in the previous examples, except the final computation, which is more complicated. So, let us begin with this computation and briefly make the connection afterwards.

\begin{lem}\label{ineqvi}
Consider the system of inequalities, with real unknowns $A_1,\dots,A_n$:
\[\begin{cases} \sum_{i=j}^nA_i\ge 0\qquad\forall j \\ A_1+A_i\ge A_{i+1}\qquad\forall i=2,\dots,n-1 \\ A_i+A_j\ge A_{i+j}+1,\qquad\forall 2\le i<j,\;i+j\le n \end{cases}\]
Then under this constraints, we have $\sum_{i=1}^n A_i\ge\lceil (n-4)/2\rceil$, and this is attained when we set $A_i=1$ for $2\le i<n/2$, and $A_i=0$ for other $i$ ($i=1$ and $n/2\le i\le n$).
\end{lem}
\begin{proof}
The given solution realizes the inequalities and the claimed minimal value of $\sum A_i$. 

Let us show the lower bound on $\sum A_i$. We begin with $n=2m-1$ odd ($m\ge 2$), so we have to show $\sum A_i\ge m-2$. Then
\begin{align*}\sum_{i=1}^{2m-1} A_i= &A_1+A_{2m-2}+A_{2m-1}+\sum_{i=2}^{m-1}(A_i+A_{2m-1-i})\\
 \ge & 2A_{2m-1}+\sum_{i=2}^{m-1}(A_{2m-1}+1)=m-2+mA_{2m-1}\ge m-2.
\end{align*}

The case when $n=2m$ is even (where we have to prove $\sum A_i\ge m-2$) is a bit more complicated as we gather terms by triples, which leads to discuss on the value of $m$ modulo 3. 

Let $i\ge 1$ be such that $m-3i\ge 2$ (so $m+3i-1\le 2m-3$). We consider the following sum of $6i$ terms
\begin{align}\sum_{j=m-3i}^{m+3i-1}A_i
	= & \sum_{k=1}^i (A_{m-3k}+A_{m-3k+1}+A_{m-3k+2}+A_{m+3k-3}+A_{m+3k-2}+A_{m+3k-1})\notag\\
		= & \sum_{k=1}^i (A_{m-3k}+A_{m+3k-2})+(A_{m-3k+1}+A_{m+3k-1})+(A_{m-3k+2}+A_{m+3k-3})\notag\\
		\ge & \sum_{k=1}^i (1+A_{2m-2})+(1+A_{2m})+(1+A_{2m-1})\notag\\
		&=3i+i(A_{2m-2}+A_{2m-1}+A_{2m}).\label{ineqqq1}
\end{align}
If instead, we choose $i$ such that $m-3i=1$, computing the previous inequality works in the same way, except that once we have to use an inequality of the form $A_1+A_j\ge A_{j+1}$ (instead of $A_1+A_j\ge A_{j+1}+1$), so there is one less $+1$ term and we get \begin{equation}\sum_{j=m-3i}^{m+3i-1}A_i\ge (3i-1)+i(A_{2m-2}+A_{2m-1}+A_{2m}).\label{ineqqq2}\end{equation}

If $m=3k+1$, we choose $i=k$, so that $(m-3i,m+3i-1)=(1,2m-2)$, $3i-1=m-2$; we then have, using (\ref{ineqqq2}) 
\[\sum_{i=1}^{2m} A_i \ge  (A_{2m-1}+A_{2m})+ m-2+k(A_{2m-2}+A_{2m-1}+A_{2m}) \ge m-2.\]

If $m=3k$, we choose $i=k-1$, so that $(m-3i,m+3i-1)=(3,2m-4)$, $3i=m-3$; using (\ref{ineqqq1}) we get.
\begin{align*}
\sum_{i=1}^{2m} A_i  \ge & (A_1+A_{2m-1})+(A_2+A_{2m-3})+A_{2m-2}+A_{2m}\\ & + m-3+(k-1)(A_{2m-2}+A_{2m-1}+A_{2m})\\
 \ge & m-2+A_{2m}+k(A_{2m-2}+A_{2m-1}+A_{2m})\ge m-2.
\end{align*}

Finally if $m=3k+2$, we choose $i=k$, so $(m-3i,m+3i-1)=(2,2m-3)$, $3i=m-2$. Then, using that
\begin{align*}
A_1+A_{2m-2}+A_{2m-1}= & \frac12((A_1+A_{2m-2})+(A_1+A_{2m-1})+(A_{2m-2}+A_{2m-1})\\
\ge &\frac12(A_{2m-1}+A_{2m}+(A_{2m-2}+A_{2m-1})\\ =&\frac12(A_{2m-2}+2A_{2m-1}+A_{2m}),
\end{align*} we get, incorporating the previous inequality (\ref{ineqqq1}),
\begin{align*}
\sum_{i=1}^{2m} A_i & \ge  A_1+A_{2m-2}+A_{2m-1}+A_{2m}+ m-2+i(A_{2m-2}+A_{2m-1}+A_{2m})\\
& \ge \frac12A_{2m}+\frac12(A_{2m-1}+A_{2m})+(i+1/2)(A_{2m-2}+A_{2m-1}+A_{2m})+m-2\\ & \ge m-2.\qedhere
\end{align*}
\end{proof}

To prove Theorem \ref{vira}, we can assume $n\ge 4$, so that $\g_{\ge i}$ are solid, as checked in Example \ref{flagfili}.

Then for convenience we rescale the basis by considering $f_i=n^{-1}e_i$. Then $[f_i,f_j]=\frac{(i-j)}{n}f_{i+j}$ for $i+j\le n$, so the coefficients are all $\le 1$, and we consider the $\ell^\infty$ norm with respect to this basis.

It follows that the bracket in $\g[r]$ is given by 
\[[f_1,f_i]=\frac{1-i}{n}f_{i+1}\;(i\ge 2);\quad [f_i,f_j]=\frac{r^{-1}(i-j)}{n}f_{i+j}\; (i,j\ge 2).\]

As in the previous cases, we consider a discrete cocompact subring $\Lambda_r$ of Guivarch systole $\ge r$ and $\Xi_r=u(r)^{-1}\Lambda_r$, which has systole $\ge 1$ and is a discrete cocompact subring in $\g[r]$. 
If $a_i$ is the systole of the projection on $\g_{\ge i}/\g_{\ge i+1}$, we consider an element $a_if_i+o(i)$ for all $i$. Computing the brackets and defining $A_i=\log_r(a_i)$, we deduce $A_1+A_i\ge A_{i+1}$ ($2\le i\le n-1$) and $A_i+A_j\ge A_{i+j}+1$ ($2\le i<j\le n-i$) (here we use that nonzero structure coefficients are $\le 1$ in absolute value). This gives rise to the system of inequalities solved in Lemma \ref{ineqvi}, and hence $\sum A_i\ge\lfloor (n-4)/2\rfloor$. Accordingly, the covolume of $\Xi_r$ is $\ge r^{\lfloor (n-4)/2\rfloor}$, and hence that of $\Lambda_r$ is $\ge r^{D+\lfloor (n-4)/2\rfloor}$.

(Note that the case $n=5$ covers $L_{5,6}$ from \S\ref{s_5dim}.)

\subsection{Further examples}
The next example, unlike the previous ones, do not admit a complete flag of solid ideals. Therefore, lower bounds on the systolic growth now rely on some geometric lemmas about lattices, such as Lemma \ref{covollati} below.

Consider the Lie algebra $\g=\g(2n+4)$ (here $2n+4$ is the dimension) with basis 
\[(U,V,W,Z,X_1,Y_1,\dots,X_n,Y_n)\] and nonzero brackets
\[[U,V]=W,\quad [U,W]=Z,\quad [X_i,Y_i]=Z, \; \forall 1\le i\le n.\] Its nilpotency length is 3, the derived subalgebra $\g^2$ is the plane generated by $(W,Z)$, $\g^3$ is the line generated by $Z$. We have $\dim(\g/\g^2)=2n+2$, and the homogeneous dimension is $D=2n+7$. Note that Proposition \ref{ineq} predicts $\sigma(r)\preceq r^{D+n+1}$. This is not sharp but almost:

\begin{thm}\label{g3nil}
For every fixed $n$, the systolic growth of $\g(2n+4)$, as a function of $r$, is $\simeq r^{3n+7}=r^{D+n}$.
\end{thm} 

The proof relies on the following general lemma.

\begin{lem}\label{covollati}
Let $V$ be a $d$-dimensional real vector space ($d$ even) with a fixed Lebesgue measure. Let $\phi$ be a symplectic form of determinant 1. Let $\Gamma$ be a lattice in $V$ such that for all $x,y\in\Gamma$ we have $\phi(x,y)\in\mathbf{Z}$. Then the covolume of $\Gamma$ is $\ge 1$.

If for some $s>0$, we have $\phi(x,y)\in s\mathbf{Z}$ for all $x,y\in\Gamma$, then the covolume of $\Gamma$ satisfies $\cov(\Gamma)\ge s^{d/2}$.
\end{lem}
\begin{proof}
We argue by induction on $d/2$. The case $d=0$ is trivial. Fix a primitive element $e_1$ in $\Gamma$. Let $H$ be its orthogonal for $\phi$ and $H'=H/\mathbf{R} e_1$. Then $\phi(e_1,\Gamma)$ is a nonzero subgroup $m\mathbf{Z}$ of $\mathbf{Z}$, with $m\ge 1$. Let $e_d$ be an element in $\Gamma$ with $\phi(e_1,e_d)=m$. 
Every element in $x\in\Gamma$ can be written in a unique way as $\lambda_x e_d+h_x$ with $(\lambda_x,hx)\in\mathbf{R}\times H$. Then $\phi(e_1,x)=\lambda_x m$ belongs to $m\mathbf{Z}$; this shows that $\lambda_x\in\mathbf{Z}$, and hence $h_x\in\Gamma$ as well. This shows that $\Gamma=(\Gamma\cap H)\oplus \mathbf{Z} e_d$. Since $e_1$ is primitive, we can write $\Gamma\cap H=\mathbf{Z} e_1\oplus \Gamma'$, where $\Gamma'$ is a lattice in a hyperplane $V'$ of $H$. We fix Lebesgue measures on $V'$ and on $(\mathbf{R} e_1\oplus \mathbf{R} e_d)$ so that the both restrictions of $\phi$ to these subspaces has determinant 1. So the product measure matches with the original Lebesgue measure on $V$. By induction, the covolume $c'$ of $\Gamma'$ in $V'$ is $\ge 1$, and the covolume of $\mathbf{Z} e_1\oplus \mathbf{Z} e_d$ in the plane it spans is $m$. So the covolume of $\Gamma$ is $c'm\ge 1$.

For the second result, we apply the hypothesis to the lattice $s^{-1/2}\Gamma$, which has covolume $\ge 1$; its covolume is also $s^{-d/2}\cov(\Gamma)$. This proves that $\cov(\Gamma)\ge s^{d/2}$.\end{proof}

\begin{proof}[Proof of Theorem \ref{g3nil}]
First define $\Lambda_r$ as the lattice with basis $rU$, $rV$, $rX_i$, $r^2Y_i$ ($1\le i\le n$), $r^2W$, $r^3Z$. Its covolume is $r^{7+3n}$. Its systole is $r$. It is a subring, by a straightforward verification. So the systolic growth is $\preceq r^{7+3n}$.

We have just described the lower central series. Also, the center is reduced to the line generated by $Z$ and the second term in the ascending central series is the codimension 2 subspace $\mk{j}$ generated by all basis vectors except $U,V$. The centralizer $\mk{c}$ of $\mk{j}$ is 3-dimensional with basis $(V,W,Z)$, and $\mk{h}=\mk{c}+\mk{j}$ is the hyperplane generated by all basis vectors except $U$. 

So we have the inclusions of solid ideals
\[\{0\}\stackrel{Z}\subset\g^3\stackrel{W}\subset\g^2\stackrel{V}\subset\mk{c}\stackrel{X_1,\dots,Y_n}\subset\mk{h}\stackrel{U}\subset\g\]

The dilation $u(r)$ is given by multiplication by $r$ on the subspace with basis $(U,V,X_1,\dots,Y_n)$, by $r^2$ on $W$ and by $r^3$ on $Z$. The nonzero brackets in $\g[r]$ are thus given by
\[[U,V]_r=W,\quad [U,W]_r=Z,\quad [X_i,Y_i]_r=r^{-1}Z, \; \forall 1\le i\le n.\]

Let a discrete cocompact subring $\Lambda_r$ have systole $\ge r$. Define $\Xi_r=u(r)^{-1}\Lambda_r$. So $\Xi_r$ intersects each of these ideals in a lattice, and the projection modulo the previous ideal, written in the canonical basis, gives generators $\alpha U$, $\beta V$, $\delta W$, $\eta Z$ (we choose the constant to be positive) and a lattice $\Gamma$ in the space with basis $(X_1,\dots,Y_n)$, elements. 

So $\Lambda_r$ contains an element $u$ of the form $\alpha U+x$ (with $x$ a combination of the other generators), and an element $v$ of the form $\beta V+tW+t'Z$, and an element $w$ of the form $\delta W+t''Z$. Then $[u,v]=\alpha\beta W+\alpha t Z$, and $[u,w]=\alpha\delta Z$. Since the intersection of $\Lambda_r$ with the plane with basis $(W,Z)$ has covolume $\delta\eta$ and the covolume of its subgroup generated by $[u,v]$ and $[u,w]$ is $\alpha^2\beta\delta$, we get $\alpha^2\beta\delta\ge\delta\eta$, or equivalently $\alpha^2\beta\ge\eta$.

Lemma \ref{bas_ineq} implies that all of $\beta\delta\eta$, $\delta\eta$, $\eta$ are $\ge 1$. It does not say that $\alpha\beta\delta\eta\ge 1$ because we have one intermediate term in the filtration, but we can deduce it:
\[(\alpha\beta\delta\eta)^2=(\alpha\beta\alpha)\beta\delta^2\eta^2\ge \eta\beta\delta^2\eta^2=(\beta\delta\eta)(\delta\eta)\eta\ge 1.\]

On the other hand, we see that the standard symplectic form $\phi$ on the subspace generated by $X_1,\dots,Y_n$ (for which $[X_i,Y_i]=\phi(X_i,Y_i)Z$) maps $\Gamma\times\Gamma$ into $r\eta\mathbf{Z}$. It follows from Lemma \ref{covollati} that $\cov(\Gamma)\ge (r\eta)^n\ge r^n$.

Therefore, the covolume of $\Xi_r$ satisfies
\[\cov(\Xi_r)=\cov(\Gamma)\alpha\beta\delta\eta\ge r^n.\]
Thus the systolic growth satisfies $\sigma(r)\ge r^{D+n}$.
\end{proof}

\begin{rem}\label{quadra_h}
The above proof can be extended to a larger class of Lie algebras, namely for $n\ge 0$ and $k\ge 3$, define $\g(k+2n,k)$ with basis $U_1,\dots,U_k$, $ X_1,Y_1,\dots,X_n,Y_n$, and nonzero brackets $[U_1,U_i]=U_{i+1}$ ($2\le i\le k-1$), $[X_i,Y_i]=U_k$. Then its systolic growth is $\simeq r\mapsto r^{D+n(k-3)}$. (For $k=5$, the nilpotency length is 4; this yields $\sigma(r)\simeq r^{D+2n}$ while Proposition \ref{ineq} predicts $\sigma(r)\preceq r^{D+2n+2})$.

If we write the dimension as $d=2n+k$ and write $h=n(k-3)$, we see that $h$ is maximal when $n=(d-3)/4$, that is, for $n$ integer, $n=(d-3+e)/4$ with $e$ an integer with $|e|\le 2$. Then for such $n$ the computation provides $h=\frac18((d-3)^2-e^2)$, with $e^2\in\{0,1,4\}$ ($e^2=0$ for $d\stackrel{4}\equiv 3$, $e^2=1$ for $d$ even, $e^2=4$ for $d\stackrel{4}\equiv 1$).

Also, if we choose $\lambda=(1+\sqrt{5})/2\sim 1.618\dots$, then we see that for $\g(n,\lfloor \lambda n\rfloor)$, we have $h=(\alpha+O(1/n))\frac{d^2}6$, with $\alpha=\frac{6\lambda}{(2+\lambda)^2}\sim 0.742\dots$. This shows that $h$ can behave as fast as the square of the dimension (we choose to write it as a factor of $\frac{d^2}6$, since $h\le \frac{d^2}6$ by Propositions \ref{ineq} and \ref{ub_kc}.
\end{rem}

\appendix

\section{Lattices and discrete cocompact subrings: back and forth}\label{app_nil}

We prove here Lemma \ref{llazard}. This lemma was mentioned to the author by Yves Benoist. While it essentially follows from Malcev's ideas, I could not find a reference with this statement precisely written. The closest I am aware of are in the books \cite[Chap.\ 6]{Seg} and \cite[\S 5.4]{corwin}:
\begin{itemize}
\item \cite[Chap.\ 6, Theorem 5]{Seg} essentially says that for every lattice $\Gamma$ is trapped between two lattices whose logarithms are additive subgroups, so that the index between the two is bounded only in terms of the dimension of the ambient Lie group.
\item In \cite[\S 5.4]{corwin}, the main result in the direction of Lemma \ref{llazard} is Proposition 5.4.8 of that book, which states that every lattice $\Gamma$ has a finite index subgroup $\Gamma'$ such that $\log\Gamma'$ is a subring (and furthermore, $\log\Gamma$ is a finite union of additive cosets of $\log\Gamma'$). Unlike in \cite{Seg}, no uniform bound is given.
\end{itemize}

These results are not enough to prove Lemma \ref{llazard} for two reasons:
\begin{itemize}
\item they do not yield anything in the direction (\ref{llaz1}) of Lemma \ref{llazard};
\item in the direction (\ref{llaz2}), they provide partial statements. In \cite{Seg} it only yields subgroups whose logarithm is an additive lattice (but maybe not a Lie subring), and in \cite{corwin} the uniformity not given.
\end{itemize}


Let us start with a general elementary fact on the covolume.

\begin{prop}\label{covol_admu}
Let $\g$ be a real fd nilpotent Lie algebra. Endow it with a choice a Lebesgue (= Haar) measure, and endow the corresponding simply connected Lie group $G$ with the push-forward of this measure by the exponential map. Then the latter is a Haar measure on $G$. Moreover, for every lattice $\Gamma$ such that $\log\Gamma$ is an additive subgroup, the covolume of $\Gamma$ in $G$ equals the covolume of $\log\Gamma$ in $\g$. In particular, for any two lattices $\Gamma\subset\Gamma'$ such that $\log\Gamma$ and $\log\Gamma'$ are both additive subgroups, the index $[\Gamma':\Gamma]$ equals the index $[\log\Gamma':\log\Gamma]$.
\end{prop}
\begin{proof}
If $\g$ is abelian, this clear. Otherwise, we argue by induction on $\dim\g$. Let $\mathfrak{n}$ be any term in the central series, distinct from $\g$ and $\{0\}$ (e.g., the derived subalgebra or the center). Fix any complement subspace $\mk{v}$ of $\mk{n}$; endow them with Lebesgue measures so that the product measure matches with the Lebesgue measure on $\g$. Write $N=\exp\mathfrak{n}$. Endow $\g/\mk{n}$ with the Lebesgue measure image of that of $\mk{v}$, and $G/N$ with its push-forward by the exponential. Denote the projections as $p:\g\to\g/\mk{n}$ and $\pi:G\to G/N$. 

Then $\pi(\Gamma)$ is a lattice, whose logarithm is $p(\log\Gamma)$. By induction, the result holds true in $\mk{n}$ and $\g/\mk{n}$. Thus, with self-explanatory notation, the covolume satisfies
\[\cov_G(\Gamma)=\cov_{N}(N\cap\Gamma)\cov_{G/N}(\pi(\Gamma))\] \[=\cov_{\mk{n}}(\mk{n}\cap\log\Gamma)\cov_{\g/\mk{n}}(p(\log\Gamma))=\cov_{\g}(\log\Gamma).\qedhere\]  
\end{proof}

We now proceed state a more precise and robust version of Lemma \ref{llazard} and then prove it.

Given a nilpotent Lie $\mathbf{Q}$-algebra, we can endow it with a group law by setting $xy=\log(\exp(x)\exp(y))$, given by the Baker-Campbell-Hausdorff formula; we still call it multiplication and denote it by $\cdot$ (or no sign). It is thus endowed with the addition, the scalar multiplication, the Lie bracket, and the multiplication. This convention is very convenient, although somewhat misleading, because $0$ is the unit of the multiplication law, the multiplication is not distributive with respect to the addition, and we have $x^n=nx$ for all $n\in\mathbf{Z}$.

Let us write the Baker-Campbell-Hausdorff formula as 
\[xy=\sum_{i\ge 1}B_{i}(x,y)=(x+y)+\frac12[x,y]+\frac1{12}([x,[x,y]]-[y,[x,y]])+\dots,\]
with $B_i$ homogeneous of degree $i$. This formula precisely makes sense in the pro-nilpotent completion of the free Lie $\mathbf{Q}$-algebra $\mk{f}$ on $(x_1,x_2)$. Let $m_i$ be the least common denominator of terms\footnote{The definition of $m_i$ is a bit sloppy, because ``least common denominator of terms" refers to some choice of basis. It can be made more rigorous as follows: let $\mk{f}=\bigoplus_{i\ge 1}\mk{f}_i$ be the standard grading of $\mk{f}$, so that $B_i\in\mathfrak{f}_i$. Let $\Lambda=\bigoplus_{i\ge 1}\Lambda_i$ be the Lie subring generated by $(x_1,x_2)$, which is a free Lie $\mathbf{Z}$-algebra on $(x_1,x_2)$. Here $\Lambda_i$ is the additive subgroup generated by brackets of length $i$ in $x_1$ and $x_2$. Define $m_i$ as the smallest positive integer $m$ such that $B_i\in m^{-1}\Lambda_i$. Now if $B_i$ is written with respect to a $\mathbf{Z}$-basis of $\Lambda_i$, then $m_i$ is indeed the smallest common denominator of the coefficients of $B_i$, and indeed the Baker-Campbell-Hausdorff is usually written with respect to such a basis (namely a Hall basis, see \cite[\S II.2.11]{Bou}).} in $B_i$. Small values are given by
\[(m_1,m_2,\dots)=(1,2,12,24,720,1440,30240\dots).\]

\begin{defn}
In a Lie $\mathbf{Q}$-algebra, we define a strong subring as an additive subgroup stable, for all $i\ge 2$, under the rescaled iterated commutator \[(x_1,\dots,x_i)\mapsto m_i^{-1}[x_1,[x_2,\dots,[x_{i-1},x_i]\cdots]].\]
\end{defn}

For $i=2$, this means the stability under $(x,y)\mapsto \frac12[x,y]$; in particular a strong subring is a Lie subring. In case the nilpotent Lie algebras, it also follows that a strong subring is closed under the group law defined by the Baker-Campbell-Hausdorff formula. There is an obvious notion of strong subring generated by a subset. (Note that the definition of strong subring also makes sense in a $c$-step nilpotent Lie algebra over a ring of characteristic some power of any prime $p>c$, because $m_i$ is coprime to $p$ for all $i<p$.)

We will need the following simple lemma.

 \begin{lem}\label{strsubgen}
Let $\g$ be a nilpotent Lie $\mathbf{Q}$-algebra. Then the strong subring generated by any finite subset is finitely generated as an additive group.
\end{lem}
\begin{proof}
The basic idea is that when ``iterating" the various rescaled brackets in the definition of strong subring, we go deeper in the central series and only a finite number of steps (thus involving bounded denominators) are necessary. 

To make this rigorous, denote by $u_2,\dots,u_c$ the above rescaled iterated brackets. These are multilinear laws. Given any two multilinear law $v,w$ on $k,\ell$ variables and $1\le i\le k$, we obtain a new multilinear law, on $k+\ell-1$ variables, by replacing, in $v(x_1,\dots,x_k)$, $x_i$ with $w(y_1,\dots,y_i)$. Note that if $k,\ell\ge 2$ then $k+\ell+1>\max(k,\ell)$. Therefore if we compute all possible iterated substitutions obtained starting from $u_2,\dots,u_c$, then we obtain only finitely multilinear laws on $\le c$ variables, say $v_1,\dots,v_q$, say with $v_i$ of degree $s_i\ge 2$. Then given any additive subgroup $A\subset\g$, the strong subring it generates is the image of the linear map 
\[\bigoplus_{i=1}^qA^{\otimes s_i}\stackrel{\bigoplus v_i}\longrightarrow\g.\]
If $A$ is finitely generated, then all these tensor powers are finitely generated as additive abelian groups, and this proves the result. (Note that we do not have to care about $u_1$ since $m_1=1$; if we had $m_1=2$, which would mean stability under division by 2, the result and its proof would fail.)

Also note that if $K$ is a field of characteristic zero and $R$ a subring, the notion of strong $R$-subring of a Lie algebra over $K$ makes sense, and the lemma can be adapted to this context with essentially no change in the proof (considering tensor products over $R$).
\end{proof}

Lemma \ref{llazard} is therefore a particular case of the following:

\begin{lem}
Let $d$ be a non-negative integer. There exists a constant $C(d)$ such that for every real nilpotent Lie algebra of dimension $\le d$, for every $\Gamma\subset G$ that is either a discrete cocompact subring in $\g$, or a lattice in the simply connected nilpotent Lie group $(\g,\cdot)$, there exist strong subrings $\Gamma_1,\Gamma_2$ in $\g$ with $\Gamma_1\subset\Gamma\subset\Gamma_2$ and such that the index $[\Gamma_2:\Gamma_1]$ (which is the same for both addition and multiplication, by Proposition \ref{covol_admu}) is finite and bounded above by $C$.
\end{lem}

\begin{proof}
Rather than a direct approach, we will rather use a general trick consisting in the idea that a non-quantitive statement for the free object mechanically provides a uniform statement. The counterpart is that the constants will not be explicit. 

Let $\g(d)$ be the free $(d-1)$-step nilpotent Lie $\mathbf{Q}$-algebra on the generators $x_1,\dots,x_d$. We also view it as a group through the Baker-Campbell-Hausdorff product. Then the multiplicative subgroup generated $x_1,\dots,x_d$ is free $(d-1)$-step nilpotent on the given generators; we denote it by $\Gamma_d$; the Lie subalgebra generated by $x_1,\dots,x_d$ is denoted by $\Lambda_d$. Then $(\Lambda_d,+)$ is a free abelian group whose rank is the $\mathbf{Q}$-dimension of $\g(d)$; in particular $\g(d)/\Lambda_d$ is a torsion group.

In $\g(d)$, the strong subring $\Xi_d$ generated by $\{x_1,\dots,x_d\}$ has a finitely generated underlying additive group, by Lemma \ref{strsubgen}.
It follows that $\Lambda_d$ has finite index in $\Xi_d$ for the additive law. Therefore, $\Xi_d$ is discrete in the real completion $\g(d)\otimes_\mathbf{Q}\mathbf{R}$ which in turn implies that the index of $\Gamma_d$ in $\Xi_d$, for the group law, is finite (because $\Gamma_d$ is a lattice for the group law).

From these two facts (one viewed in the Lie algebra side, one in the group side), we obtain that there exists $n=n(d)\ge 1$ such that for all $x\in\Xi_d$, $nx\in\Lambda_d$ and $x^n\in\Gamma_d$. But since $x^n=nx$, this just means that $n\Xi_d\subset \Lambda_d\cap\Gamma_d$; clearly $n\Xi_d$ is also a strong subring; its index in $\Xi_d$ as an additive subgroup is $n^d$; its index as a multiplicative subgroup is also $n^d$, by Proposition \ref{covol_admu}.

Now let $\g$ be an arbitrary real nilpotent Lie algebra, of dimension $d$; then it is $(d-1)$-step-nilpotent. First, let $\Gamma$ be a lattice for the group law (recall that this means that $\exp(\Gamma)$ is a lattice, but we keep our identification). Then $\Gamma$ being an iterated extension of $d$ cyclic groups, it has a generating family $(y_1,\dots,y_d)$ as a group. We can then map the free $(d-1)$-step nilpotent group into it mapping $x_d\mapsto y_d$, and this extends to a group homomorphism $f:(\g(d),\cdot)\to(\g,\cdot)$. Then this is also a Lie algebra homomorphism $\g(d)\to\g$. It follows that the image of any strong subring is a strong subring. So $f(n\Xi_d)\subset\Gamma\subset f(\Xi_d)$, and $[f(n\Xi_d):f(\Xi_d)]\le n(d)^d$.

The same argument can be performed with discrete cocompact subrings (exchanging the roles of groups and Lie rings), completing the proof of the lemma.
\end{proof}


\end{document}